\documentclass[11pt,reqno]{amsart}
\usepackage{amsmath,amsthm,amssymb,enumerate}
\usepackage{hyperref}
\usepackage{mathrsfs}
\usepackage{nicefrac}
\usepackage{mhequ}

\usepackage{xfrac}
\usepackage{color}

\numberwithin{equation}{section}

\newtheorem{theorem}{Theorem}[section]
\newtheorem{lemma}[theorem]{Lemma}

\newtheorem*{theorem*}{Theorem}

\theoremstyle{definition}
\newtheorem{assumption}{Assumption}[section]
\newtheorem{definition}{Definition}[section]

\theoremstyle{remark}
\newtheorem{remark}{Remark}[section]

\newcommand\bR{\mathbb{R}}
\newcommand\bP{\mathbb{P}}

\newcommand{\bN}{\mathbb{N}}

\newcommand*{\beq}{\begin{equation}}
\newcommand*{\eeq}{\end{equation}}
\newcommand{\E}{\mathbb{E}}

\newcommand{\bit}{\begin{itemize}}
\newcommand{\eit}{\end{itemize}}
\newcommand\D{\partial}

\DeclareMathOperator*{\esssup}{ess\,sup}

\begin{document}

\title[$L_\infty$-estimates for degenerate SPDE]{Supremum estimates for degenerate, quasilinear stochastic partial differential equations}

\author[K. Dareiotis and B. Gess]{Konstantinos Dareiotis and Benjamin Gess}

\address[K. Dareiotis]{Max Planck Institute for Mathematics in the Sciences, Inselstrasse 22, 04103 Leipzig, Germany}
\email{konstantinos.dareiotis@mis.mpg.de}

\address[B. Gess]{Max Planck Institute for Mathematics in the Sciences, Inselstrasse 22, 04103 Leipzig and Faculty of Mathematics,
University of Bielefeld, 
33615 Bielefeld, 
Germany}
\email{benjamin.gess@gmail.com}

\begin{abstract}
We prove a priori estimates in $L_\infty$ for a class of quasilinear stochastic partial differential equations. 
The estimates are obtained independently of the ellipticity constant $\varepsilon$ and thus imply analogous estimates for degenerate quasilinear stochastic partial differential equations, such as the stochastic porous medium equation.
\end{abstract}
\maketitle

\section{Introduction}
We consider quasilinear stochastic partial differential equations (SPDEs) of the form\footnote{Throughout the article we use the summation convention with respect to integer valued repeated indices.}
\begin{equation}      \label{eq: quasi}
\begin{aligned}                   
du &=\left[ \D_i \left( a^{ij}_t(x,u)\D_j u  +F^i_t(x,u)\right)+ F_t(x,u)  \right] \, dt 
\\
& + \left[  \D_i \big( g^{ik}_t(x,u) \big)+G^k_t(x, u) \right] \, d \beta^k_t,
\\
u_0&= \xi,
\end{aligned}
\end{equation}
for $(t,x) \in [0,T] \times Q=:Q_T$, with zero Dirichlet conditions on $\D Q$, for some bounded open set $Q \subset \bR^d$ and $\beta^k$ being independent Wiener processes. 

In this work, using Moser's iteration techniques (see e.g.\ \cite{MOS}),  we prove the following: First, roughly speaking, we show that if the initial condition $\xi$ is in $L_\infty$ then the solution is in $L_\infty$ for all times $t \geq 0$. Second, we show a regularizing effect, that is, if the initial condition is in $L_2$, then for all $t >0$ the solution $u(t)$ is in $L_\infty$ and the corresponding norm blows up at a rate of $t^{-\tilde{\theta}}$, for some constant $\tilde{\theta}>0$, as $t \searrow 0$.

A key point in these results is that the obtained estimates are uniform with respect to the ellipticity constant of the diffusion coefficients $a^{ij}$ and thus can be applied to the case of degenerate, quasilinear SPDE, such as the porous medium equation.

More precisely,  under certain conditions on the coefficients $a^{ij}, F^i, F, g^{ik}, G^k$ (see Assumptions \ref{as: nd}-\ref{as:boundednessV} below for details), we prove the following $L^\infty$ bounds: 
\begin{theorem*}[see Theorems \ref{thm: quasilinear nd} and \ref{thm: smoothing2}]     
Let $\alpha>0$, {\color{black}$\mu \in [2, \infty] \cap ((d+2)/2,\infty]$}. There exists constants $N$, $\tilde \theta>0$ such that if $u$  is a solution of  \eqref{eq: quasi}, then 
\begin{equs}    
\E \|u \|^\alpha _{L_\infty(Q_T)} \ \leq N \E \left( 1+\|\xi \|_{L_\infty(Q)}^\alpha {\color{black}+\|V^1\|_{L_\mu(Q_T)}^\alpha+ \|V^2\|_{L_{2\mu}(Q_T)}^\alpha } \right),
\end{equs}
and 
\begin{equation*}         
\E \|u \|^\alpha _{L_\infty((\rho, T) \times Q)} \ \leq \rho^{-\tilde{\theta}} N \E \left( 1+\|\xi \|_{L_2(Q)}^\alpha {\color{black}+\|V^1\|_{L_\mu(Q_T)}^\alpha+ \|V^2\|_{L_{2\mu}(Q_T)}^\alpha} \right),
\end{equation*}
for all $\rho \in (0,T)$.

\end{theorem*}
{\color{black}In the above theorem $V^1$ and $V^2$ are functions that can be regarded as dominating any existing ``free terms" coming from the drift part and the noise part of the equation, respectively (cf. Assumption \ref{as: nd} below).}

A key point in the two estimates above is that the constants $N$ and $\tilde{\theta}$ are independent of the ellipticity constant of the diffusion coefficients $a^{ij}$. Hence, the established estimates carry over without change to degenerate SPDE such as stochastic porous media equations
\begin{equation}               \label{eq: PME stratonovich}
\begin{aligned}
du = & \left[ \Delta \left( |u|^{m-1}u \right) +f_t(x) \right]  \, dt + \sum_{i=1}^d \sigma_t\D_i u \circ d\tilde{\beta}^i_t\\
+&  \sum_{k=1}^\infty\left[\nu^k_t(x) u  +g^k_t(x) \right] dw^k_t
\\
u_0= & \xi,
\end{aligned}
\end{equation}
with zero Dirichlet conditions on $\D Q$  and  $m \in (1, \infty)$, where $\tilde{\beta}^1_t,...\tilde{\beta}^d_t, w^1_t,w^2_t,...$  are independent $\bR$-valued standard  Wiener processes. The corresponding theorem reads as follows:
\begin{theorem*}[see Theorems \ref{thm: theorem boundedness} and \ref{thm: smoothing PME}]
Let ${\color{black}\mu \in [2, \infty] \cap ((d+2)/2,\infty]}$. There exists constants $N$, $\tilde \theta>0$ such that if $u$  is a solution of 
  \eqref{eq: PME stratonovich}, then 
\begin{equation*}          
\E \|u \|^2 _{L_\infty(Q_T)} \ \leq N \E \left( 1+\|\xi \|_{L_2(Q)}^2{\color{black} +\|f\|^2_{L_\mu(Q_T)}+ \||g|_{l_2}\|^2_{L_{2\mu}(Q_T)}}\right),
\end{equation*}
and 
\begin{align*}          
\E \|u \|^2 _{L_\infty((\rho, T) \times  Q)} \ &\leq \rho^{-\tilde{\theta}} N \E \left(1+\| \xi \|_{H^{-1}}^2 {\color{black}+\|f\|^2_{L_\mu(Q_T)}+ \||g|_{l_2}\|^2_{L_{2\mu}(Q_T)}} \right),
\end{align*}
for all $\rho \in (0,T)$. 
\end{theorem*}

 We restrict to affine operators in the noise in \eqref{eq: PME stratonovich} for the sole reason that no complete well-posedness theory in $L_p$ spaces of \eqref{eq: PME stratonovich} with non-linear noise is yet available. We emphasize that this linear structure is \textit{not} required in the derivation of the a priori bounds established in this work. Concerning the well-posedness for nonlinear noise we also refer the reader to \cite{BEN2} for a well-posedness theory of such equations in a kinetic framework. 
 
In the following we will briefly comment on existing literature on the regularity of solutions to stochastic porous media equations. The existence of strong solutions (i.e.\ $|u|^{m-1}u \in L^2((0,T);H^1_0$) has been shown in \cite{BEN1} under the assumption that the operators in the noise are bounded and Lipschitz continuous and under the assumption that $\xi \in L_{m+1}$. In the case of linear multiplicative noise (and $\sigma=0$) \eqref{eq: PME stratonovich} can be transformed into a PDE with random coefficients. Based on this, the H\"older-continuity and boundedness of solutions has been shown in \cite{G13-2,BAR}.

Concerning the regularity theory for deterministic singular and degenerate quasilinear equations we refer to \cite{CAF,DIB2,SACKS} (see also the monographs \cite{DIB,VAS} and the references therein). The regularity of solutions to non-degenerate SPDE has been addressed in \cite{DENIS,HOF,KOMA,KOMA2,WANG, MATE}. For general background on SPDE and
stochastic evolution equations we refer to \cite{KR,MR,BAR2,PARDOUX}.

\subsection{Notation}
Let us introduce some notation that will be used throughout this paper. Let $T$ be a positive real number. Let $(\Omega, \mathcal{F}, \mathbb{F}, \bP)$ be a filtered probability space, where  the filtration $\mathbb{F}=( \mathcal{F}_t)_{t \in [0,T]}$ is right continuous and $\mathcal{F}_0$ contains all $\bP$-null sets. We assume that on $\Omega$ we are given a sequence of independent one-dimensional $\mathbb{F}$ -Wiener processes $(\beta^k_t)_{k=1}^\infty$. The predictable $\sigma$-field on $\Omega_T:= \Omega \times [0,T]$ will be denoted by $\mathcal{P}$.  Let $Q \subset \bR^d$ be a bounded open domain. For $t \in [0,T]$ we set $Q_t= [0,t] \times Q$. The norm in $L_p(Q)$ will be denoted by $\|\cdot\|_{L_p}$. We denote by $H^1_0$  the completion of $C^\infty_c(Q)$ under the norm 
$$
\|u\|_{H^1_0}^2 := \int_Q |\nabla u |^2 \, dx 
$$
 and by $H^{-1}$ the dual of $H^1_0(Q)$.  
 For $q \in [1, \infty)$, we denote by $\mathbb{H}_q^{-1}$ the set of all $H^{-1}$-valued,  $\mathbb{F}$-adapted, continuous processes $u$, such that $u \in L_q(\Omega_T, \mathcal{P}; L_q(Q))$.  Similarly, we denote by $\mathbb{L}_2$ the set of all $L_2(Q)$-valued,  $\mathbb{F}$-adapted, continuous processes $u$, such that $u \in L_2(\Omega_T, \mathcal{P}; H^1_0(Q))$. We will write $(\cdot, \cdot)_H$ for the inner product in a Hilbert space $H$.  
For $m \geq 1$, we will consider the Gel'fand triple 
$$
L_{m+1}(Q) \hookrightarrow H^{-1}\hookrightarrow (L_{m+1}(Q))^*.
$$
 The duality pairing  between $L_{m+1}(Q)$ and  $(L_{m+1}(Q))^*$ will be denoted by 
 ${}_{L_{m+1}^*}\langle \cdot, \cdot \rangle_{L_{m+1}}$. Notice that this duality is defined by means of the inner product in $H^{-1}$. Consequently,  for $u,v \in C^\infty_c(Q)$ 
 $$
 {}_{L_{m+1}^*}\langle u, v \rangle_{L_{m+1}} =(u,v)_{H^{-1}}= (u, (-\Delta)^{-1}v)_{L_2(Q)} \neq \int_Q u v \, dx.
 $$
 For more details we refer to \cite[pp. 68-70]{MR}.
  We will use the summation convention with respect to integer valued repeated indices. Moreover, when no confusion arises, we suppress the $(t,x)$-dependence of the functions for notational convenience.
 
The article is organized in two sections. In Section 2 we  prove our results for the non-degenerate equation. In Section 3, we verify the well-posedness of the degenerate equation, and we approximate the solution by the method of the vanishing viscosity, and by using the estimates of the previous section we pass to the limit.

\section{Non-Degenerate Quasilinear SPDE}

As already mentioned in the introduction, in order to obtain the desired estimates for equation \eqref{eq: PME stratonovich} we first study a class of non-degenerate SPDEs. More precisely,
we consider SPDEs of the form 
\begin{align}                   
\nonumber
du &=\left[ \D_i \left( a^{ij}_t(x,u)\D_j u  +F^i_t(x,u)\right)+ F_t(x,u)  \right] \, dt 
\\
\label{eq: nd quasilinear}
& + \left[  \D_i \big( g^{ik}_t(x,u) \big)+G^k_t(x, u) \right] \, d \beta^k_t,
\\
\label{eq: initial nd quasilinear}
u_0&= \xi,
\end{align}
for $(t,x) \in [0,T] \times Q$, with zero Dirichlet conditions on $\D Q$.

\begin{assumption}                             
 \label{as: nd}
$$
$$
\vspace{-1cm}
\begin{enumerate}[(i)]
\item The functions $a^{ij}, F^i, F : \Omega_T \times Q \times \bR \to \bR$ are $\mathcal{P}\otimes \mathcal{B}(Q)\otimes \mathcal{B}(\bR) $-measurable.

\item  The functions $g^i, G : \Omega_T \times Q \times \bR \to l_2$ are $\mathcal{P}\otimes \mathcal{B}(Q)\otimes \mathcal{B}(\bR) $-measurable.

\item There exists constants $c > 0$, $\theta>0$  and $\tilde{m} >0$ such that for all $(\omega,t,x,r)\in \Omega_T \times Q \times \bR$ 
\begin{equation}                \label{eq: ellipticity}
(a^{ij}_t(x,r) - \frac{1}{2}\D_rg_t^{ik}(x,r)\D_rg_t^{jk}(x,r)) \xi^i \xi^j \geq (c  | r|^{\tilde{m}}+\theta) |\xi|^2.
\end{equation}

{\color{black}\item \label{item: regularity gi}For all $(\omega,t) \in \Omega_T$ we have $F^i_t \in C^1(\bar{Q} \times \bR)$, $g_t^i\in C^2( \bar{Q} \times \bR;l_2)$. Moreover, there exist  predictable processes $ V^1, V^2: \Omega_T \to L_2(Q)$, such that  $V^1, V^2 \in  L_2((0,T); L_2(Q))$ almost surely, and a constant $K$ such that for all $(\omega, t, x, r)\in \Omega_T \times Q \times \bR$ 
\begin{align}
\label{eq: growth condition,F}
|F_t(x,r)|+|F^i_t(x,r)|+| \D_i F^i_t(x,r)|& \leq  V^1_t(x)+ K |r|
\\   \label{eq: growth condition}
|G_t(x,r)|_{l_2}+|g^i_t(x,r)|_{l_2}+|\D_ig^i_t(x,r)|_{l_2}+ |\D^2_ig^i_t(x,r)|_{l_2}&\leq V^2_t(x)+ K |r|
\\
 \label{eq: boundedness of coef}
|\D_rg^i_t(x,r)|_{l_2}+|\D_r\D_i g^i_t(x,r)|_{l_2} &\leq K.
\end{align}}

\item \label{item: regularity g} Let $
\mathbb{N}_g := \{k \in \bN : \exists i \in \{1,...,d\}, g^{ik} \not\equiv 0\}$. We assume in addition that for all $(\omega,t)$ we have $(G^k_t)_{k \in \mathbb{N}_g} \in C^1( \bar Q \times \bR; l_2(\mathbb{N}_g))$, and for all $(\omega,t,x,r) \in \Omega_T \times Q\times \bR $ 
{\color{black}\begin{equation}                 \label{eq: one extra derivative}
\sum_{k \in \mathbb{N}_g}  |\D_i G^k_t(x,r)|^2
\leq V^2_t(x)+ K |r|.
\end{equation}}

\item The initial condition $\xi $ is an $\mathcal{F}_0$-measurable $L_2(Q)$-valued random variable.
\end{enumerate}
\end{assumption}

{\color{black}\begin{assumption}            \label{as:boundednessV}
There exists a constant $N$ such that for all $(\omega, t, x) \in  \Omega_T \times Q$ we have 
\begin{equs}         \label{eq:boundednessV}
|V^1_t(x)|+|V^2_t(x)| \leq N.
\end{equs}
\end{assumption}}
{\color{black}\begin{remark}
Assumption \ref{as:boundednessV} is purely technical in the sense that  our estimates do not depend on the bound of $V^1$ and $V^2$. As it will be seen in the next section, Assumption 2.2 can be removed provided that one has solvability of the equation and some stability properties with respect to the ``free-terms". 
\end{remark}}

\begin{definition}                    \label{def: definition L2}
A function $u \in \mathbb{L}_2$ will be called a solution of \eqref{eq: nd quasilinear}-\eqref{eq: initial nd quasilinear} if 
\begin{enumerate}[(i)]
\item For all $i,j \in \{1,...,d\}$, almost surely  
$$
 \int_0^T \| a^{ij}(u) \D_iu \|^2_{L_2} \, dt < \infty.
$$
\item For all $\phi \in H^1_0$, almost surely, for all $t \in [0,T]$ 
\begin{align*}
(u_t, \phi )_{L_2} =(\xi , \phi )_{L_2} & + \int_0^t\left[ \left( F(u), \phi \right)_{L_2} - \left( a^{ij}(u) \D_j u  +F^i (u), \D_i \phi  \right)_{L_2} \right]  dt
\\
&+ \int_0^t \left[ \big( G^k(u), \phi \big)_{L_2}-\big( g^{ik}(u), \D_i \phi \big)_{L_2} \right] \, d \beta^k_t,
\end{align*}
\end{enumerate}
\end{definition}

 We first present a collection of lemmas that will be used in the proofs of the main theorems. 
The following can be found in \cite{MY} (see Proposition IV.4.7 and Exercise IV.4.31/1).
\begin{lemma}\label{lem: Revuz-Yor}

Let $X$ be a non-negative, adapted, right-continuous process, and let $Y$ be a non-decreasing, adapted, continuous process such that
$$
\E (X_{\tau}| \mathscr{F}_0 )\leq \E (Y_{\tau}| \mathscr{F}_0)
$$
for any bounded stopping time $\tau \leq T$. Then for any $\sigma\in(0,1)$
$$
\E \sup_{t\leq T}X_t^{\sigma}\leq \sigma^{-\sigma}(1-\sigma)^{-1}\E Y_T^{\sigma}.
$$

\end{lemma}
The following lemma is well known (see, e.g.,  \cite[p.8, Proposition 3.1]{DIB}). We provide the proof in order to emphasize that the constant $C$ can be chosen 
independent of $\lambda$ for $\lambda \in [1,2]$ (see below).
\begin{lemma} \label{lem: embeding}
There exists a constant $N$ such that for all  $\lambda \in [1,2]$, $s \in [0,T]$  and all $v \in L_\infty((s,T); L_\lambda(Q))\cap L_2((s,T) ; H^1_0(Q))$, we have
\begin{equation}          \label{eq:emb}
\int_s^T \int_Q |v|^q\, dx dt \leq N^q \left( \int_s^T \int_Q |\nabla v|^2 \, dx dt \right) \left( \esssup_{s\leq t \leq T} \int_Q |v|^\lambda \, dx \right)^{2/d},
\end{equation}
where $q=q(\lambda)=2(d+\lambda)/d$.
\end{lemma}

\begin{proof}
By the Gagliardo-Nirenberg inequality (see \cite[p.62, Theorem 2.2]{LAD}) we have (notice that $d(2-\lambda)+2\lambda>0$) for a.e. $t \in (0,T)$
$$
\|v_t\|_{L_q} \leq N(\lambda)\| \nabla v_t\|^{2/q}_{L_2} \|v_t\|^{(q-2)/q}_{L_\lambda}
$$
where 
$$
N(\lambda):= \left( I_{1=d}\frac{1+\lambda}{\lambda}+I_{d=2}\max\left\lbrace\frac{q(d-1)}{d}, \frac{\lambda+2}{2} \right\rbrace+I_{d>2} \frac{2(d-1)}{d-2}     \right)^{2/q}.
$$
Since $C:=\sup_{\lambda \in [1,2]} N(\lambda) < \infty$, the result follows by taking the $q$-th power in the inequality above and integrating over $(0,T)$.  
\end{proof}

Next is It\^o's formula for the $p$-th power of the $L_p$ norm. It can be proved as \cite[Lemma 2]{KOMA}
with the help of a localization argument.

\begin{lemma}                \label{lem: ito formula}
Let Assumption \ref{as: nd} hold and let $u$ be a solution of \eqref{eq: nd quasilinear}. Moreover, suppose that for some $p \geq 2$ and some $s \in [0,T)$, almost surely 
{\color{black}$$
\| u_s\|^p_{L_p}+ \int_s^T \left( \| V^1\|_{L_p}^p+  \| V^2\|_{L_p}^p \, \right) dt < \infty.
$$}
Then, almost surely 
\begin{equation}                      \label{eq: finite energy}
\sup_{s \leq t \leq T} \|u_t\|_{L_p}^p+   \int_s^T \int_Q (|u|^{p-2}+|u|^{\tilde{m}+p-2}) | \nabla u|^2 \, dx dt  < \infty.
\end{equation}
Moreover, almost surely 
\begin{align}
\nonumber
&\| u_t\|_{L_p}^p= \|u_s\|_{L_p}^p+ \int_s^t \int_Q \left( A_{ij}(x,u)\D_j u  +F^i(x,u)\right) p(1-p)|u|^{p-2} \D_i u  \,dxdz 
\\
\nonumber
&+ \int_s^t \int_Q \left(\frac{1}{2} p(p-1)| \D_i(g^i(x,u))+G(x,u)|_{l_2}^2|u|^{p-2}+ p F(x,u)u |u|^{p-2} \right)  dx  dz
\\
\label{eq: Ito formula}
&+ \int_s^t\int_Q \left(  p(1-p) g^{ik}(x,u) |u|^{p-2}  \D_iu+pG^k(x,u) u |u|^{p-2} \right) \,dx d\beta^k_z,
\end{align}
for all $t \in [s,T]$.
\end{lemma}

{\color{black}From now on we fix $\mu \in \Gamma_d:=  [2, \infty]\cap ((d+2)/2, \infty]$, we denote by $\mu'$ its conjugate exponent, that is, $
\frac{1}{\mu}+\frac{1}{\mu'}=1$, 
and we set 
\begin{align*}
\gamma&:= 1+ (2/d),
\\
\bar \gamma & := \gamma / \mu',
\\
\mathfrak{N}&:=\{ l \in [2, \infty) : l=\tilde{m}(1+\bar \gamma+...+\bar \gamma^n)/ \mu', \   n \in \mathbb{N} \},
\\
\kappa & := \sup_{p \in \mathfrak{N}} \max\{2p/(p-1), I_{p \neq 2} 4p/(p-2)\} < \infty.
\end{align*}
Notice that $\bar \gamma >1$. }
\begin{lemma}            \label{lem: right right}
Let Assumptions \ref{as: nd}-\ref{as:boundednessV} hold and let $u $ be a solution of \eqref{eq: nd quasilinear}. Then, for all $p \in \mathfrak{N}$, $q \geq p$, and $\eta \in (0,1)$ we have         
\begin{align}
\nonumber
&\E \left( A_q \vee \left( \sup_{t \leq T} \|u_t\|^p_{L_p(Q)} +\int_0^T \int_Q\left| \nabla |u|^{(\tilde{m}+p)/2} \right|^2 \, dx dt \right) \right)^\eta
\\
\label{eq: right right estimate}
&{\color{black}\leq \frac{\eta^{-\eta}}{1-\eta} (N p^\kappa)^\eta  \E  \left( \left(A_q \vee  \| u\|_{L_{\mu'p}(Q_T)}^p \right)+p^{-p}\left(   \|V^1\|_{L_\mu(Q_T)}^p+ \|V^2\|_{L_{2\mu}(Q_T)}^p\right) \right)^\eta,}
\end{align}
where $A_q=(1+ \|\xi\|_{L_\infty})^q$, and $N$ is a constant depending only on  $\tilde{m}, T, c,K,$ $d, \mu$, and $|Q|$. 
\end{lemma}

\begin{proof}
We  assume that the right hand side in \eqref{eq: right right estimate} is finite, or else there is nothing to prove. Under this assumption, for each $p \in \mathfrak{N}$ we have the formula \eqref{eq: Ito formula} with $s=0$. We proceed by estimating the terms that appear at the right hand side of \eqref{eq: Ito formula}.
We have 
$$
F^i(x,u)p(1-p)|u|^{p-2} \D_i u= \D_i (\mathcal{R}_p(F^i)(x,u))- \mathcal{R}_p(\D_iF^i)(x,u),
$$
where for a function $f$ we have used the notation 
$$
\mathcal{R}_p(f)(x,r)= \int_0^rp(p-1) f(x,s)|s|^{p-2} \, ds.
$$
Moreover, from \eqref{eq: finite energy}, \eqref{eq: growth condition,F}, the fact that $V^1$  is bounded and the definition of $\mathcal{R}_p(F^i)(x,u)$, it follows   (see Lemma \ref{lem:Appendix}) that $\mathcal{R}_p(F^i)(\cdot,u) \in W^{1,1}_0(Q)$ for a.e.  $(\omega, t)\in \Omega_T$, which in particular implies that 
$$
\int_Q \D_i (\mathcal{R}_p(F^i)(\cdot,u)) \, dx =0.
$$
Moreover, one can see from \eqref{eq: growth condition,F} that 
{\color{black}\begin{align*}
| \mathcal{R}_p(\D_iF^i)(x,r)| &\leq p V^1(x) |r|^{p-1}+K(p-1)|r|^p
\end{align*}}
By H\"older's inequality and Young's inequality, we obtain
{\color{black}\begin{equs}
&  \ \ \ \int_0^t \int_Q | \mathcal{R}_p(\D_iF^i)(x,u)| \, dx ds 
\\
& \leq p \|V^1\|_{L_\mu(Q_t)} \|u\|^{p-1}_{L_{\mu'(p-1)}(Q_t)}+K(p-1) \|u\|_{L_{p}(Q_t)}^p
\\
&\leq N p \|V^1\|_{L_\mu(Q_t)} \|u\|^{p-1}_{L_{\mu'p}(Q_t)}+K(p-1) \|u\|_{L_{\mu'p}(Q_t)}^p
\\
& \leq N p^{-p} \|V^1\|^p_{L_\mu(Q_t)} 
+N(Kp+p^{2p/(p-1)}) \|u\|_{L_{\mu'p}(Q_t)}^p.
\end{equs}}
 Consequently, almost surely, for each $t \in [0,T]$
\begin{equation}    \label{eq: estimate F}             
\int_0^t \int_Q F^i(x,u)p(1-p)|u|^{p-2} \D_i u \, dx ds {\color{black}\leq N p^{-p}\|V^1\|_{L_\mu(Q_t)}^p + Np^\kappa\|u\|_{L_{\mu'p}(Q_t)}^p,}
\end{equation}
where $N$ depends only on $K$ and $|Q|$. We continue with the estimate of the term 
$$
\frac{1}{2} p(p-1) \int_0^t\int_Q | \D_i(g^i(x,u))+G(x,u)|_{l_2}^2 |u|^{p-2}\, dx ds.
$$
Obviously, 
\begin{align*}
& \int_Q | \D_i(g^i(x,u))+G(x,u)|_{l_2}^2 |u|^{p-2} \, dx \\
=&\sum_{k \in \mathbb{N}_g^c}\int_Q|G^k(x,u)|^2|u|^{p-2}\, xd+
 \sum_{k \in \mathbb{N}_g}\int_Q| \D_i(g^{ik}(x,u))+G^k(x,u)|^2 |u|^{p-2} \, dx.
\end{align*}
By the growth condition \eqref{eq: growth condition}, H\"older's inequality and Young's inequality
we have 
\begin{equs}
& \frac{1}{2}p(p-1)\sum_{k \in \mathbb{N}_g^c}\int_0^t \int_Q |G^k(x,u)|^2|u|^{p-2} \, dx ds
\\
 \leq & {\color{black}N  p^{-p}\|V^2\|^p_{L_{2\mu}(Q_t)}+ Np^\kappa \|u\|^p_{L_{\mu'p}(Q_t)}.}
\end{equs}
Moreover, by Assumption \ref{as: nd} \eqref{item: regularity g} we have
\begin{align*}
&\sum_{k \in \mathbb{N}_g}| \D_i(g^{ik}(x,u))+G^k(x,u)|^2
\\
=&\sum_{k \in \mathbb{N}_g}|\D_rg^{ik}(x,u)\D_iu|^2+ \sum_{k \in \mathbb{N}_g}| \D_ig^{ik}(x,u)+G^k(x,u)|^2
\\
+&\sum_{k \in \mathbb{N}_g}2\D_rg^{ik}(x,u)\D_iu(\D_ig^{ik}(x,u)+G^k(x,u)).
\end{align*}
By the growth condition \eqref{eq: growth condition}, H\"older's inequality and Young's inequality we have
\begin{align*}
&\frac{1}{2}p(p-1)\sum_{k \in \mathbb{N}_g}\int_0^t \int_Q | \D_ig^{ik}(x,u)+G^k(x,u)|^2|u|^{p-2} \, dx ds\\
 \leq & {\color{black}N  p^{-p}\|V^2\|^p_{L_{2\mu}(Q_t)}+ Np^\kappa \|u\|^p_{L_{\mu'p}(Q_t)}}.
\end{align*}
Moreover, we have 
\begin{align*}
&p(p-1)\sum_{k\in \mathbb{N}_g } \D_rg^{ik}(x,u)\D_iu(\D_ig^{ik}(x,u)+ G^k(x,u))|u|^{p-2}
\\
&=\D_i\left( \mathcal{R}_p\left( \mathfrak{g}\right) (x,u)\right)-  \mathcal{R}_p\left( \D_i \mathfrak{g}\right) (x,u),
\end{align*}
where 
$$
\mathfrak{g}:= \sum_{k\in \mathbb{N}_g} \D_rg^{ik}( \D_ig^{ik}+G^k).
$$
As before, it follows that  $\mathcal{R}_p\left( \mathfrak{g}\right) (\cdot,u) \in W^{1,1}_0(Q)$, which in turn implies that 
$$
\int_Q \D_i (\mathcal{R}_p\left( \mathfrak{g}\right) (x,u)) \, dx =0.
$$ 
By \eqref{eq: growth condition}, \eqref{eq: boundedness of coef}, and \eqref{eq: one extra derivative},  we have 
$$
\int_0^t \int_Q |\mathcal{R}_p\left( \D_i \mathfrak{g}\right) (x,u)|\, dx ds \leq {\color{black}N  p^{-p}\|V^2\|^p_{L_{2\mu}(Q_t)}+ Np^\kappa \|u\|^p_{L_{\mu'p}(Q_t)}}.
$$

Consequently, almost surely, for all  $t \in [0,T]$ we have 
\begin{align}
\nonumber
&\frac{1}{2} p(p-1) \int_0^t\int_Q | \D_i(g^i(x,u))+G(x,u)|_{l_2}^2 |u|^{p-2}\, dx ds
\\
\nonumber
\leq & \frac{1}{2} p(p-1)\int_0^t \int_Q |\D_rg^i(x,u) \D_iu |^2_{l_2} |u|^{p-2}\, dx ds
\\   \label{eq: estimate quadratic variation}
+ & {\color{black}N  p^{-p}\|V^2\|^p_{L_{2\mu}(Q_t)}+ Np^\kappa \|u\|^p_{L_{\mu'p}(Q_t)}},
\end{align}
where $N$ depends only on $K$ and $|Q|$. In a similar manner one gets 
$$
p \int_Q F(x,u)u |u|^{p-2} \, dx {\color{black}\leq N  p^{-p}\|V^1\|^p_{L_\mu(Q_t)}+ Np^\kappa \|u\|^p_{L_{\mu'p}(Q_t)}.}
$$
Using the above inequality combined with \eqref{eq: estimate F}, \eqref{eq: estimate quadratic variation}, and \eqref{eq: ellipticity} we obtain from \eqref{eq: Ito formula}
\begin{align}
\nonumber
&\| u_t\|_{L_p}^p + cp(p-1)\int_0^t \int_Q|u|^{\tilde{m}+p-2}| \nabla u|^2 \, dx ds 
\\
\label{eq: consequence of Ito's}
\leq & \|\xi\|_{L_p}^p + {\color{black}N \left(p^{-p}\|V^1\|^p_{L_\mu(Q_t)}+p^{-p}\|V^2\|^p_{L_{2\mu}(Q_t)}+ Np^\kappa \|u\|^p_{L_{\mu'p}(Q_t)} \right)} +M_t,
\end{align}
where $M_t$ is the local martingale from \eqref{eq: Ito formula}. For any stopping time $ \tau \leq T$ and  any $B \in \mathcal{F}_0$ we have by the Burkholder-Davis-Gundy inequality 
\begin{align*}
\E \sup_{t \leq \tau} I_B|M_t| & \leq 
N \E I_B \left( \int_0^\tau \sum_k \left( p(1-p)\int_Q g^{ik}(x,u) |u|^{p-2}  \D_iu \, dx \right)^2 \, ds \right)^{1/2}
\\
&+N \E I_B \left( \int_0^\tau \sum_k \left(p\int_Q G^k(x, u)u |u|^{p-2}\, dx \right)^2 \, ds \right)^{1/2}.
\end{align*}
We have 
$$
 p(1-p) g^{ik}(x, u) |u|^{p-2}  \D_iu =\D_i( \mathcal{R}_p ( g^{ik} ) (x,u))- \mathcal{R}_p ( \D_i g^{ik} ) (x,u).
$$
As before, we have $\mathcal{R}_p ( g^{ik} ) (\cdot,u) \in W^{1,1}_0(Q)$, which implies that 
$$
\int_Q \D_i(\mathcal{R}_p ( g^{ik} ) (x,u)) \, dx =0.
$$
Next notice that by Minkowski's integral inequality,  H\"older's inequality, and Young's inequality, we have
\begin{align}
\nonumber
\sum_k \left( \int_Q  \mathcal{R}_p ( \D_i g^{ik} ) (x,u) \, dx \right)^2 &\leq \left( \int_Q | \mathcal{R}_p ( \D_i g^{ik} ) (x,u) |_{l_2} \, dx \right)^2 
\\
\nonumber
&\leq   \left(  \int_Q \int_{-|u|}^{|u|}p(p-1) |\D_ig^i(x,s)|_{l_2} |s|^{p-2} \, ds dx \right)^2
\\
\nonumber
&\leq N \left( 2p  \int_Q {\color{black}|V^2(x)|}|u|^{p-1}+|u|^p \, dx \right)^2
\\
\label{eq:quadratic-estimate}
&\leq N \|u\|_{L_p}^p \left( p^2\int_Q{\color{black} |V^2(x)|^2 }|u|^{p-2} \, dx + p^2 \|u\|^p_{L_p}\right),
 \end{align}
 {\color{black}which implies 
 \begin{equs}  \label{eq: quadratic variation}
 & \int_0^t \sum_k \left( \int_Q  \mathcal{R}_p ( \D_i g^{ik} ) (x,u) \, dx \right)^2  \, ds 
 \\
\leq &  N \sup_{s \leq t}  \|u_s\|_{L_p}^p \left( p^{-p}\|V^2\|^p_{L_{2\mu}(Q_t)}+ p^\kappa \|u\|^p_{L_{\mu'p}(Q_t)} \right).
 \end{equs}}
 Consequently, by Young's inequality we have for any $\varepsilon>0$ 
 \begin{align}
 \nonumber
& N \E I_B \left( \int_0^\tau \sum_k \left( p(1-p)\int_Q g^{ik}(x,u) |u|^{p-2}  \D_iu \, dx \right)^2 \, ds \right)^{1/2} 
 \\
 \nonumber
 &\leq \varepsilon \E I_B \sup_{t \leq \tau}\|u_t\|_{_p}^p + \frac{1}{\varepsilon}N \E I_B  {\color{black}\left( p^{-p}\|V^2I_{[0, \tau]}\|^p_{L_{2\mu}(Q_T)}+ p^\kappa \|uI_{[0, \tau]}\|^p_{L_{\mu'p}(Q_T)} \right) }.
 \end{align}
 In a similar manner, for any $\varepsilon>0$ we get 
 \begin{align*}
&N \E I_B \left( \int_0^\tau \sum_k \left(p\int_Q G^k(x,u) u |u|^{p-2} \, dx \right)^2 \, ds \right)^{1/2}
\\
&\leq \varepsilon \E I_B \sup_{t \leq \tau}\|u_t\|_{_p}^p + \frac{1}{\varepsilon}N \E I_B {\color{black} \left( p^{-p}\|V^2I_{[0, \tau]}\|^p_{L_{2\mu}(Q_T)}+ p^\kappa \|uI_{[0, \tau]}\|^p_{L_{\mu'p}(Q_T)} \right) .}
\end{align*}
Hence, we obtain for any $\varepsilon>0$
\begin{equs}  
\E \sup_{t \leq \tau}I_B| M_t| & \leq  \varepsilon \E\sup_{t \leq \tau}I_B\|u_t\|_{_p}^p
\\
 + \frac{1}{\varepsilon}N \E I_B  & \left(  p^{-p}\|V^2I_{[0, \tau]}\|^p_{L_{2\mu}(Q_T)}+ p^\kappa \|uI_{[0, \tau]}\|^p_{L_{\mu'p}(Q_T)} \right).  \label{eq: estimate martingale}
\end{equs}
By \eqref{eq: consequence of Ito's} we have 
\begin{align}        
\nonumber  
&\E I_B \sup_{t \leq \tau} \| u_t\|_{L_p}^p \leq  \E I_B\|\xi\|_{L_p}^p  +\E I_B \sup_{t \leq \tau}|M_t|
\\
+N &\E I_B  {\color{black}\left( p^{-p}\|V^1I_{[0, \tau]}\|^p_{L_\mu(Q_T)}+ p^{-p}\|V^2I_{[0, \tau]}\|^p_{L_{2\mu}(Q_T)}+ p^\kappa \|uI_{[0, \tau]}\|^p_{L_{\mu'p}(Q_T)} \right).}
\label{eq:sup-u} 
\end{align}
By \eqref{eq: consequence of Ito's} again, 
 after a localization argument we obtain
\begin{align}
\nonumber
&\frac{4cp(p-1)}{(p+\tilde{m})^2}\E I_B \int_0^ \tau \int_Q \left| \nabla |u|^{(p+ \tilde{m})/2} \right|^2 \, dx ds \leq N \E I_B  \|\xi\|_{L_p}^p
\\
&  + {\color{black}N \E I_B  \left( p^{-p}\|V^1I_{[0, \tau]}\|^p_{L_\mu(Q_T)}+ p^{-p}\|V^2I_{[0, \tau]}\|^p_{L_{2\mu}(Q_T)}+ p^\kappa \|uI_{[0, \tau]}\|^p_{L_{\mu'p}(Q_T)} \right)}
\label{eq: estimate deriv of power}
\end{align}
and notice that for all $p \geq 2$
$$
\frac{4cp(p-1)}{(p+\tilde{m})^2} \geq N(\tilde{m}, c),
$$
and therefore it can be dropped from the right hand side of \eqref{eq: estimate deriv of power}.
Let us denote by $\tau_n$ the first exit time of $\|u_t\|^p_{L_p}+ {\color{black}\| V^1 \|^p_{L_\mu(Q_t)} +\| V^2 \|^p_{L_{2\mu}(Q_t)}}$ from $(-n,n)$, and by $C_n:= \{ \|\xi\|_{L_p} \leq n\}$. For an arbitrary $C \in \mathcal{F}_0$ and an arbitrary stopping time $\rho \leq T$, we apply \eqref{eq: estimate martingale} with $\tau= \tau^n \wedge \rho=: \rho_n$ and $B= C \cap C_n=:H_n$, which combined with \eqref{eq:sup-u} gives after rearrangement 
\begin{equs}
&\E\sup_{t \leq \rho_n}I_{H_n}\|u_t\|_{L_p}^p \leq N \E I_{H_n} \| \xi \|_{L_p}^p 
\\
 + N &{\color{black}\E I_{H_n}  \left( p^{-p}\|V^1I_{[0, \rho_n]}\|^p_{L_\mu(Q_T)}+ p^{-p}\|V^2I_{[0, \rho_n]}\|^p_{L_{2\mu}(Q_T)}+ p^\kappa \|uI_{[0, \rho_n]}\|^p_{L_{\mu'p}(Q_T)} \right).}
\end{equs}
By the above inequality and \eqref{eq: estimate deriv of power} (applied with $\tau= \rho_n$, $B=H_n$) one can easily see that for all $q \geq p$ we have 
$$
\E I_C X^{n, q}_\rho \leq \E I_C Y^{n,q}_\rho < \infty,
$$
where 
\begin{align*}
X_t^{n,q}&:=I_{C_n}\left( A_q \vee \left( \sup_{s \leq \tau_n \wedge t} \|u_s\|^p_{L_p} +\int_0^{t \wedge\tau_n} \int_Q\left| \nabla |u|^{(\tilde{m}+p)/2} \right|^2 \, dx ds \right) \right) 
\\
Y_t^{n,q}&:=  N p^\kappa I_{C_n}{\color{black}\left(  \left( A_q \vee\|uI_{[0, t \wedge\tau_n]}\|^p_{L_{\mu'p}(Q_T)}\right) \right. }
\\
& {\color{black}\left.   + p^{-p}\left( \|V^1I_{[0, t \wedge\tau_n]}\|^p_{L_\mu(Q_T)}+ \|V^2I_{[0, t \wedge\tau_n]}\|^p_{L_{2\mu}(Q_T)}\right) \right)}
\end{align*}
and $A_q=(1+ \|\xi\|_{L_\infty})^q$.
By Lemma \ref{lem: Revuz-Yor} we have 
$$
\E \left( X^{n, q}_T \right)^\eta  \leq \frac{\eta^{-\eta}}{1-\eta} \E \left(  Y^{n,q}_T \right)^\eta.
$$
The assertion now follows by letting $n \to  \infty$. 

\end{proof}

\begin{lemma}                  \label{lem: right right local}
Let Assumptions \ref{as: nd}-\ref{as:boundednessV} hold and let u be a solution of \eqref{eq: nd quasilinear}. Let $\rho \in (0,1)$ and set $r_n=\rho(1-2^{-n})$.  Then for all $p \in \mathfrak{N}$, $\eta \in (0,1)$, and $n \in  \mathbb{N}$  we have 
\begin{equs}
& \E\left( \sup_{t \in [r_{n+1},T]}\|u\|^p_{L_p} + \int_{r_{n+1}}^T \int_Q \left|\nabla |u|^{(\tilde{m}
+p)/2} \right|^2 \, dx dt \right)^\eta \leq \Big(N p^\kappa \frac{2^n}{\rho}\Big)^\eta \frac{\eta^{-\eta}}{1-\eta}  
\\
 & \times {\color{black}\E \left(  \| uI_{[r_n, T]}\|_{L_{\mu'p}(Q_T)}^p +p^{-p}\left(   \|V^1\|_{L_\mu(Q_T)}^p+ \|V^2\|_{L_{2\mu}(Q_T)}^p\right)\right) ^\eta,}
 \\
 \label{eq: right right local}
\end{equs}
where $N$ is a constant depending only on  $\tilde{m}, T, c,K,d, \mu$ and $|Q|$. 
\end{lemma}

\begin{proof}
We assume that the right hand side of \eqref{eq: right right local} is finite and we set 
$$
c_n =\rho\left( 1- \frac{3}{4}2^{-n} \right).
$$
There exists a $t' \in (r_n , c_n)$ such that almost surely 
$$
\|u_{t'}\|_{L_p}^p+ {\color{black}\int_{t'}^T( \|V^1_s\|_{L_p}^p+\|V^2_s\|_{L_p}^p) \, ds< \infty}.
$$
 Let $ \psi \in C^1([0, T])$ with $0 \leq \psi \leq 1$, $\psi_t=0$ for $ 0\leq t \leq c_n$, $\psi_t=1$ for $r_{n+1}\leq t \leq T$, and $|\psi'_t| \leq 2^{n+2} \rho^{-1}$. By Lemma \ref{lem: ito formula} we obtain 
\begin{align}
\nonumber
&\psi_t\| u_t\|_{L_p}^p=  p(1-p)\int_0^t \int_Q \psi \left( A_{ij}(u)\D_j u  +F^i(u)\right) |u|^{p-2} \D_i u  \, ds 
\\
\nonumber
&+ \int_0^t \psi \left[ \frac{1}{2} p(p-1)\int_Q| \D_i(g^i(u))+G(u)|_{l_2}^2 |u|^{p-2}\, dx+ p \int_Q F(u)u |u|^{p-2} \, dx \right] \, ds
\\
\nonumber 
&+ \int_0^t  \psi \left[ p(1-p)\int_Q g^{ik}(u) |u|^{p-2}  \D_iu\, dx +p\int_Q G^k(u) u |u|^{p-2}\, dx  \right] \, d\beta^k_s
\\
\label{eq: ito product}
&+\int_0^t \| u_s \|^p_{L_p} \psi' \, ds. 
\end{align}
By using the estimates obtained in the proof of Lemma \ref{lem: right right}, we obtain 
\begin{equs}
\nonumber
&\psi_t\| u_t\|_{L_p}^p + cp(p-1)\int_0^t \int_Q\psi  |u|^{\tilde{m}+p-2}| \nabla u|^2 \, dx ds 
\\
& {\color{black}\leq  N p^{-p} \left( \|V^1\|_{L_\mu(Q_t)}^p + \|V^2\|_{L_{2\mu}(Q_t)}^p\right)+ N  p^\kappa \| \psi^{1/p}   u\|^p_{L_{\mu'p}(Q_t)} }
\\
\label{eq: consequence of Ito's local}
& +\int_0^t  \psi'\|u\|^p_{L_p}  ds +M_t,
\end{equs}
where $M_t$ is the local martingale from \eqref{eq: ito product}. From this, by using arguments almost identical to the ones of the proof of Lemma \ref{lem: right right} one gets
\begin{equs}
\nonumber 
&\E\left( \sup_{t \in [0,T]}\psi_t \|u_t\|^p_{L_p} + \int_0^T \int_Q \psi \left| \nabla |u|^{(\tilde{m}
+p)/2} \right|^2 \, dx dt \right)^\eta  
\\
 \leq & {\color{black}N^\eta \frac{\eta^{-\eta}}{1-\eta}  
\times  \E  \left( p^{-p} \left( \|V^1\|_{L_\mu(Q_T)}^p + \|V^2\|_{L_{2\mu}(Q_T)}^p\right) \right.}
\\
+ &  {\color{black} \left. p^\kappa \|  \psi^{1/p}u\|^p_{L_{\mu'p}(Q_T)}  + \int_0^T \psi' \|u\|_{L_p}^p  \, ds \right)^\eta.}
\end{equs}
Having in mind that {\color{black}$\mu'p>p$} and that $p^\kappa+|\psi'| \leq 8p^\kappa 2^n\rho^{-1}$, the result follows from the properties of $\psi$.  
\end{proof}

\begin{lemma}                                 \label{lem: stochastic Gronwal}
Let Assumption \ref{as: nd} hold and let $u$ be a solution of \eqref{eq: nd quasilinear}. Then for all $p \geq 2$ and all $\alpha>0$, we have 
$$
\E \sup_{t \leq T} \|u_t\|_{L_p}^\alpha \leq N \E\|\xi\|_{L_p}^ \alpha + {\color{black}N \E \left( \int_0^T \|V^1 \|_{L_p}^p + \|V^2 \|_{L_p}^p \, ds \right)^{\alpha/p}},  
$$
where $N$ depends on $\alpha, p, K, T$ and $d$.
\end{lemma}

\begin{proof}
We assume that the right hand side is finite. {\color{black}Similarly to  \eqref{eq: consequence of Ito's}, one can show that}
\begin{align}
\nonumber
\| u_t\|_{L_p}^p  \leq  \|\xi\|_{L_p}^p 
+ N {\color{black}\int_0^t \left( \|V^1\|_{L_p}^p+ \|V^2\|_{L_p}^p+ \|u\|^p_{L_p} \right)  ds} +M_t,
\end{align}
where $M_t$ is the local martingale from \eqref{eq: Ito formula}. Moreover, as in the derivation of \eqref{eq:quadratic-estimate}, one can check that 
$$
\langle M\rangle_t \leq N \int_0^t\left(  \|u\|_{L_p}^p{\color{black}\|V^2\|^p_{L_p} }+ \|u\|_{L_p}^{2p} \right) \, ds.
$$
The result then follows from Lemma 5.2 in \cite{GGK}. 

\end{proof}

We are now ready to present our first main result.
\begin{theorem}                        \label{thm: quasilinear nd}
Let Assumptions \ref{as: nd}-\ref{as:boundednessV} hold and let $u \in \mathbb{L}_2$  be a solution of  \eqref{eq: nd quasilinear}-\eqref{eq: initial nd quasilinear}. Then, {\color{black}for all $\mu \in \Gamma_d $}, $\alpha >0$,  we have 
\begin{equation}          \label{eq: estimate nd quasilinear}
\E \|u \|^\alpha _{L_\infty(Q_T)} \ \leq N \E \left( 1+\|\xi \|_{L_\infty(Q)}^\alpha +{\color{black}\|V^1\|_{L_\mu(Q_T)}^\alpha +\|V^2\|_{L_{2\mu}(Q_T)}^\alpha }\right),
\end{equation}
where $N$ is a constant depending only on $\alpha, \tilde{m}, T, c,K,d, \mu$ and $|Q|$. 
\end{theorem}

\begin{proof}
{\color{black}We fix $\alpha>0$, $\mu\in \Gamma_d$, and let $\mu'$ be the conjugate exponent of $\mu$.  Without loss of generality we assume that the right hand side of \eqref{eq: estimate nd quasilinear} is finite.  Recall also the notations $\gamma= 1+(2/d)$, $\bar\gamma= \gamma / \mu'(>1)$ and let $\delta:= \tilde{m}\bar\gamma / (\bar\gamma-1)$. }
By Lemma \ref{lem: embeding}, after raising \eqref{eq:emb}
to the power $\gamma^{-1}$, we obtain by Young's inequality (with exponents $p=\gamma$, $p^*=\gamma/(\gamma-1)$, and note that $ 2/ d(\gamma-1)=1$)
\begin{equation}              \label{eq:C}
\left(\int_s^T \int_Q |v|^q\, dx dt\right)^{1/\gamma} \leq C^{q/\gamma} \left( \int_s^T \int_Q |\nabla v|^2 \, dx dt + \esssup_{s\leq t \leq T} \int_Q |v|^\lambda \, dx \right).
\end{equation}
For $p \geq \tilde{m}$, we apply this inequality with 
 $\lambda= 2p/(\tilde{m}+p) \in [1,2]$, $q=2(1+(\lambda/d))$, $v=|u|^{(\tilde{m}+p)/2}$ (notice that $q(\tilde{m}+p)/2= \tilde{m}+\gamma p{\color{black}= \tilde{m}+\mu'\bar \gamma p}$), and we raise to the power {\color{black}$\alpha (\mu')^{n+1}/ (\delta \gamma^n)$} to conclude that
\begin{align}
\nonumber
&{\color{black}\E \left( A_\alpha \vee  \left( \int_0^T \int_Q |u|^{\tilde{m}+\mu'p \bar \gamma} \, dx dt  \right) ^{\color{black}{\alpha /(\delta\bar\gamma^{n+1})}} \right)} \\
\nonumber
=&\E \left( A_\alpha \vee  \left( \int_0^T \int_Q |u|^{\tilde{m}+p \gamma} \, dx dt  \right) ^{\color{black}{\alpha (\mu')^{n+1} /(\delta\gamma^{n+1})}} \right) \\
\nonumber
\leq&  N^{1/{\color{black}\bar \gamma^n}} \E \left( A_\alpha \vee \left( \sup_{t \leq T} \|u_t\|^p_{L_p} +\int_0^T \int_Q \left| \nabla |u|^{(\tilde{m}+p)/2} \right|^2 \, dx dt \right)^{\color{black}{\alpha (\mu')^{n+1}/ (\delta\gamma^n)}} \right)
\\
\label{eq: right estimate}
\leq & N^{1/{\color{black}\bar \gamma^n}} \E \left( A_{\delta \bar\gamma^n/\mu'} \vee \left( \sup_{t \leq T} \|u_t\|^p_{L_p} +\int_0^T \int_Q \left| \nabla |u|^{(\tilde{m}+p)/2} \right|^2 \, dx dt \right)\right)^{\color{black}{\alpha \mu'/ (\delta\bar \gamma^n)} }
\end{align}
where, recall that for $q \geq 0$, $A_q:=(1+\|\xi\|_{L_\infty})^q$.
Let  
$$
{\color{black}p_n:= \tilde{m}(1+...+\bar \gamma^n)= \frac{\tilde{m}(\bar \gamma^{n+1}-1)}{\bar \gamma-1},}
$$ and 
let $n_0$ be the minimal positive integer such that 
$$
{\color{black}p_{n_0} \geq 2\mu' \ \ \text{and} \ \ \alpha \mu' / (\delta \bar \gamma^{n_0} ) < 1.}
$$
By combining inequality \eqref{eq: right estimate} ({\color{black}with $p=p_n/\mu'$})
with \eqref{eq: right right estimate} (with {\color{black}$\eta= \alpha \mu' \delta^{-1}\bar \gamma^{-n}$}, {\color{black}$q= \delta \bar \gamma^n /\mu' \geq p_n/\mu'$}) we obtain for all $n \geq n_0$
{\color{black}\begin{align}
\nonumber
& \E \left( A_{\alpha} \vee  \left( \int_0^T \int_Q |u|^{p_{n+1}} \, dx dt  \right) ^{\color{black}{\alpha  /(\delta \bar \gamma^{n+1})}} \right) \\
\nonumber
& =\E \left( A_{\alpha} \vee  \left( \int_0^T \int_Q |u|^{\color{black}{\tilde{m}+\bar \gamma p_n}} \, dx dt  \right) ^{\color{black}{\alpha   /(\delta \bar \gamma^{n+1})}} \right)
\\
\nonumber
& \leq c_n    \E  \left(\left(A_{\delta \bar \gamma^n/\mu'} \vee 
\|u\|^{p_n/\mu'}_{L_{p_n}(Q_T)}  \right) \right.
\\ \nonumber
 &\left.
+\left( \frac{p_n}{\mu'}\right)^{-p_n/\mu'} \left(   \|V^1\|_{L_\mu(Q_T)}^{p_n/\mu'}+ \|V^2\|_{L_{2\mu}(Q_T)}^{p_n/\mu'}\right) \right)^{\alpha \mu' /(\delta \bar \gamma^n)}
\\
\nonumber
& \leq c_n  \E   \left(A_{\alpha } \vee  \left(\int_0^T \int_Q |u|^{p_n} \,dxds \right)^{\alpha /(\delta \bar \gamma^n)} \right)
\\
\label{eq: iterable}
&+c_n \left( \frac{p_n}{\mu'}\right)^{-\alpha p_n/(\delta \bar \gamma^n)} \E \left(2+ \|V^1\|_{L_\mu(Q_T)}^\alpha+\|V^2\|_{L_{2\mu}(Q_T)}^\alpha\right), 
\end{align}}
where 
{\color{black}\begin{equation} \label{eq: cn}
c_n:= N^{1/ {\color{black}\bar\gamma^n}} (\delta {\color{black}\bar\gamma^n}/ (\mu'\alpha) )^{\alpha \mu' /(\delta {\color{black}\bar \gamma^n})} \frac{1}{1-(\alpha \mu' /(\delta {\color{black}\bar \gamma^n}))} \left(N \frac{p_n^\kappa}{(\mu')^k}\right)^{\alpha \mu'/(\delta {\color{black}\bar \gamma^n})},
\end{equation}}
$N$ does not depend on $n$, and we have used that $p_n/{\color{black}\bar \gamma^n  }\uparrow \delta$.  Notice that the right hand side of \eqref{eq: iterable} is finite (by the assumption that the right hand side of \eqref{eq: estimate nd quasilinear} is finite, Lemma \ref{lem: stochastic Gronwal} and \eqref{eq:boundednessV}). One can easily see that 
{\color{black}$$
\prod_{n=1}^\infty N^{1/{\color{black}\bar \gamma^n}}< \infty, \qquad \prod_{n=1}^\infty  (\delta{\color{black} \bar \gamma^n}/ (\mu' \alpha))^{\alpha \mu' /( \delta {\color{black}\bar \gamma^n})} < \infty.
$$}
Also,   
${\color{black}p_n \leq \tilde{m}n \bar \gamma^n}$ which implies that 
{\color{black}$$
\prod_{n=1}^\infty  \left(N \frac{p_n^\kappa}{(\mu')^k}\right)^{\alpha \mu' /(\delta {\color{black}\bar \gamma^n})}< \infty.
$$}
Moreover, since $e^{-2x} \leq 1-x$ for all $x$ sufficiently small, we have for some constant $N$ that and all $M \in \mathbb{N}$
{\color{black}$$
\prod_{n=1}^M \frac{1}{1-(\alpha \mu' / (\delta {\color{black}\bar \gamma^n}))} \leq N  e^{\sum_{n=1}^M 2\alpha \mu' / \delta {\color{black}\bar \gamma^n}},
$$}
 which implies that {\color{black}$\prod_{n=1}^\infty \frac{1}{1-(\alpha \mu' / (\delta {\color{black}\bar \gamma^n}))}  < \infty$}.  
Consequently,  there exists an $N \in \bR$ such that for any $M \in \mathbb{N}$ 
\begin{equation}      \label{eq:bound-cn}
 \prod_{n=1}^M c_n \leq N.
\end{equation}
Since $p_n/ {\color{black}\bar \gamma^n }\uparrow \delta$, there exists an $N$ such that for all $n \in \mathbb{N}$ large enough, we have 
 $$
 p_n^{- \alpha p_n /(\delta  {\color{black}\bar \gamma^n})} \leq N ({\color{black}\bar\gamma^n })^{-\alpha p_n/(\delta{\color{black}\bar\gamma^n)}} \leq  N\left({\color{black}\bar\gamma}^{ \alpha /2}\right)^{-n}.
 $$
 Since, $\alpha>0$, $
 {\color{black}\bar \gamma}>1$  we have 
 \begin{equation}                      \label{eq:bound-pn}
 \sum_{n=1}^\infty\left(\frac{ p_n}{{\color{black}\mu'}}\right) ^{-\alpha p_n/ (\delta{\color{black}\bar  \gamma^n})
 }  < \infty.
 \end{equation}
Consequently, by iterating \eqref{eq: iterable} and using \eqref{eq:bound-cn} we obtain 
\begin{equation}            \label{eq:before-limit-iter}
\Theta_m \leq  \left( \prod_{n=n_0}^m c_n\right) \Theta_{n_0} + N\left( \sum_{n=n_0}^m \lambda_n \right)  \E \left(1+ {\color{black}\mathcal{V}_\mu}\right)^\alpha,
\end{equation} 
where 
$$
\Theta_n := \E \left( A_{\alpha} \vee  \left( \int_0^T \int_Q |u|^{p_{n}} \, dx dt  \right) ^{\alpha /(\delta {\color{black}\bar \gamma^{n})} }\right), \qquad \lambda_n: = \left( \frac{p_n}{\mu'} \right) ^{-\alpha p_n/ (\delta {\color{black}{\color{black}\bar \gamma^n}})
 },
$$ 
and 
$$
{\color{black}\mathcal{V}_\mu:= \|V^1\|_{L_\mu(Q_T)}+\|V^2\|_{L_{2\mu}(Q_T)}.}
$$
By virtue of \eqref{eq:bound-cn} and \eqref{eq:bound-pn}, we can let $m \to \infty$ in \eqref{eq:before-limit-iter} and use that $ p_m/(\delta{\color{black} \bar \gamma^m}) \to 1$ to obtain by  Fatou's lemma
 \begin{align*}
 \E \| u\|_{L_\infty(Q_T)}^\alpha & \leq \liminf_{n\to \infty} \E \left( \int_0^T \int_Q |u|^{p_n} \, dx dt  \right) ^{\alpha /(\delta {\color{black}\bar  \gamma^n})} 
 \\
 & {\color{black}\leq  N \E \| u\|_{L_{p_{n_0}}(Q_T)}^{\alpha p_{n_0}/( \delta \bar \gamma^{n_0})} + N \E \left(1+ \|\xi\|_{L_\infty(Q)}^\alpha+|\mathcal{V}_\mu|^\alpha\right).}
 \end{align*}
 {\color{black}Since $p_n / \bar \gamma^n $ is increasing in $n$, we have $p_{n_0} / (\delta \bar \gamma^{n_0}) \leq 1 $ and thus
\begin{equation}        \label{eq: estimate with n0}
\E \| u\|_{L_\infty(Q_T)}^ {\alpha}\leq  N \E \left(1+\| u\|_{L_{p_{n_0}}(Q_T)}^{\alpha}+ \|\xi\|_{L_\infty(Q)}^\alpha+|\mathcal{V}_\mu|^\alpha\right).
\end{equation}}
{\color{black}Notice that by the assumption that the right hand side of \eqref{eq: estimate nd quasilinear} is finite, combined with \eqref{eq:boundednessV} and Lemma \ref{lem: stochastic Gronwal}, we get that the right hand side of \eqref{eq: estimate with n0} is finite.  
By the interpolation inequality
$$
\|u\|_{L_{p_{n_0}}(Q_T)} \leq \varepsilon\|u\|_{L_\infty(Q_T)} +N_\varepsilon \|u\|_{L_2(Q_T)},
$$
we obtain after rearrangement in \eqref{eq: estimate with n0}
$$
\E \| u\|_{L_\infty(Q_T)}^ {\alpha}\leq  N \E \left(1+\| u\|_{L_2(Q_T)}^{\alpha}+ \|\xi\|_{L_\infty(Q)}^\alpha+|\mathcal{V}_\mu|^\alpha\right),
$$
 which again by virtue of Lemma \ref{lem: stochastic Gronwal} gives (since $\mu \geq 2$)
\begin{equation*}       
 \E \| u\|_{L_\infty(Q_T)}^ {\alpha}\leq  N \E \left(1+ \|\xi\|_{L_\infty(Q)}^\alpha+|\mathcal{V}_\mu|^\alpha \right).
\end{equation*}
This finishes the proof.}
\end{proof}

\begin{remark}
In \cite{KOMA}, in the non-degenerate case, mixed $L^t_\nu L^x_\mu$-norms of the free terms appear at the right hand side of the estimates. This is also achievable in our setting provided that one has a mixed-norm version of the embedding Lemma \ref{lem: embeding} (see also \cite[Lemma 1]{KOMA}).
\end{remark}

{\color{black}Next we present the ``regularizing" effect. Recall that $\gamma= 1+(2/d)$, $\bar \gamma= \gamma/\mu'$, $\delta= \tilde{m} \bar \gamma / (\bar \gamma-1)$, and  $p_n = \tilde{m}(1+\bar \gamma+...+ \bar \gamma^n)$. We will need the following two lemmata.}
\begin{lemma}              \label{lem:p0}
Let $\alpha>0$, and let $q:=p_{n_0}$, where $n_0$ is the minimal positive integer such that 
$p_{n_0} \geq 2$ and $\alpha / ( \delta \bar \gamma ^{n_0})< 1$.  Suppose that Assumptions \ref{as: nd}- \ref{as:boundednessV} are satisfied and let $u$ be a solution of \eqref{eq: nd quasilinear}. Then, for all $ \rho \in (0,1)$ we have
\begin{equation}         \label{eq: smoothing}
{\color{black}\E \|u \|^\alpha _{L_\infty((\rho, T) \times Q)} \ \leq \rho^{-\tilde{\theta}}  N \E \left( 1+ \|u \|^\alpha _{L_{q}((r_{n_0},T)\times Q)} + |\mathcal{V}_\mu|^\alpha \right),}
\end{equation}
where 
$$
r_{n_0}=\rho(1-2^{-n_0}), \qquad {\color{black}|\mathcal{V}_\mu|= \|V^1\|_{L_\mu(Q_T)}+\|V^2\|_{L_{2\mu}(Q_T)}}, 
$$
$N$ is a constant depending only on $\alpha, \tilde{m}, T, c,K,d, \mu$, and $|Q|$, and $\tilde{\theta} > 0$ is a constant depending only on $\alpha, d, \mu$ and $\tilde{m}$.
\end{lemma}

\begin{proof}
Similarly to the proof of Theorem \ref{thm: quasilinear nd}, by Lemma \ref{lem: embeding} and Lemma \ref{lem: right right local}, we have for all $n \geq n_0$
{\color{black}\begin{align}
\nonumber
& \E \left(  \int_{r_{n+1}}^T \int_Q |u|^{p_{n+1}} \, dx dt  \right) ^{\alpha/(\delta \bar \gamma^{n+1})}
\\
\nonumber
&\leq (\rho^{-1}2^n)^{\mu' \alpha/(\delta \bar \gamma ^n)} c_n  \E   \left(\int_{r_n}^T \int_Q |u|^{p_n} \, ds \right)^{\alpha/(\delta \bar \gamma^n)} 
\\
\label{eq: iterable local}
&+(\rho^{-1}2^n)^{\alpha \mu'/(\delta \bar \gamma ^n)} c_n \left( \frac{p_n}{\mu'} \right) ^{-\alpha p_n/(\delta \bar \gamma^n)} \E \left(2+\|V^1\|_{L_\mu(Q_T)}^\alpha+\|V^2\|_{L_{2\mu}(Q_T)}^\alpha \right), 
\end{align}}
where 
$c_n$ is given in \eqref{eq: cn}. Under the assumption that the right hand side of \eqref{eq: smoothing} is finite, it follows that the right hand side of the above inequality it is also finite for $n=n_0$, and by the same inequality and induction it follows that is finite for all $n \geq n_0$. Also notice that for all $M \in \mathbb{N}$ 
$$
\prod_{n=n_0}^M (\rho^{-1}2^n)^{\color{black}{\alpha \mu'/(\delta \bar \gamma ^n )}} \leq N \rho^{-\tilde{\theta}},
$$
with $\tilde{\theta}= (\alpha {\color{black}\mu'} / \delta) \sum_{n}{\color{black}\bar \gamma}^{-n}$.
Consequently, by iterating \eqref{eq: iterable local} and passing to the limit as $n \to \infty$ we obtain
\begin{align*}
\E \|u \|^\alpha _{L_\infty((\rho, T) \times Q)} \ & \leq \rho^{-\tilde{\theta}} N \E \left(\int_{r_{n_0}}^T \int_Q |u|^{p_{n_0}} \, ds \right)^{\sfrac{\alpha}{\delta {\color{black}\bar \gamma^{n_0}}}}\\
&+ \rho^{-\tilde{\theta}} N \left(1+ {\color{black}| \mathcal{V}_\mu|^\alpha}\right)
\\ 
&\leq \rho^{-\tilde{\theta}} N \E \left( 1+ \|u \|^\alpha _{L_{p_{n_0}}((r_{n_0},T)\times Q)} + {\color{black}| \mathcal{V}_\mu|^\alpha }\right),
\end{align*}
where we have used that {\color{black}$p_{n_0} \leq \delta \bar \gamma^{n_0}$. }
\end{proof}

\begin{lemma}       \label{thm: smoothing}
Suppose that Assumptions \ref{as: nd}-\eqref{as:boundednessV} are satisfied, let $\alpha >0$,  and let $u \in \mathbb{L}_2$  be a solution of  \eqref{eq: nd quasilinear}-\eqref{eq: initial nd quasilinear}. Then, for all  $\rho \in (0,1)$ we have 
\begin{align}          \label{eq:up-to-L2-1}
\E \|u \|^\alpha _{L_\infty((\rho, T) \times  Q)} \ &\leq \rho^{-\tilde{\theta}} N \E \left( 1+\| u \|_{L_2(Q_T)}^\alpha +{\color{black}|\mathcal{V}_{\mu}|^\alpha} \right),
\end{align}
where  $N$ is a constant depending only on $\alpha, \tilde{m}, T, c,K,d, \mu$, and $|Q|$, and $\tilde{\theta}>0$  is a constant depending only on $\alpha, d,\mu $ and $\tilde{m}$.
\end{lemma}
\begin{proof}
Due to Lemma \ref{lem:p0}, we only need to estimate 
$\E \|u \|^\alpha _{L_{p_{n_0}}((r_{n_0},T)\times Q)} $ by the right hand side of \eqref{eq:up-to-L2-1}. For this, it suffices to show that 
for all $\beta>0$, $p>2$, and $\varrho \in (0,1)$ we have 
\begin{equation}          \label{eq:up-to-L2}
\E \|u \|^\beta _{L_p((\varrho, T) \times  Q)} \ \leq N \varrho^{-\tilde{\theta}} \E \left( 1+\| u \|_{L_2(Q_T)}^\beta +{\color{black}|\mathcal{V}_\mu|}^\beta \right),
\end{equation}
where $N$ is a constant depending only on $\beta, p, \varrho, \tilde{m}, T, c,K,d, r$ and $|Q|$, and $\tilde{\theta}>0$ depends only on $\beta,  d$ and $\tilde{m}$.
{\color{black}We assume that the right hand side of \eqref{eq:up-to-L2} is finite.  Let us set $p_0=2$, $p_{n+1}=\tilde{m}+ p_n \bar \gamma$, $n'=\min \{ n\in \mathbb{N} \colon p_n \geq p \}$, and 
$\varrho_k= k\varrho / n'$, for $k=0,...,n'$.  Clearly, it suffices to prove that for all $k=0,...,n'$ we have 
\begin{equation}       \label{eq:rk}
\E \|u \|^\beta _{L_{p_{k+1}}((\varrho_{k+1}, T) \times  Q)} \ \leq \varrho^{-\tilde \theta } N \E \left( 1+\| u \|_{L_{p_k}((\varrho_k,T)\times Q)}^\beta +|\mathcal{V}_\mu|^\beta \right),
\end{equation}}
since \eqref{eq:up-to-L2} follows by iterating \eqref{eq:rk} finitely many times. We assume that 
the right hand side of \eqref{eq:rk} is finite and we first prove it for {\color{black}$k\geq 1$}. Let $\varrho_k'=(\varrho_k+\varrho_{k+1})/2$. Let $ \psi \in C^1([0, T])$ with $0 \leq \psi \leq 1$, $\psi_t=0$ for $ 0\leq t \leq \varrho_k'$, $\psi_t=1$ for $\varrho_{k+1}\leq t \leq T$, and $|\psi'_t| \leq 2 n' \varrho^{-1}$. 
Then, similarly to \eqref{eq: consequence of Ito's} we have for $p \geq 2$
\begin{align}
\nonumber
&\psi_t\| u_t\|_{L_p}^p + cp(p-1)\int_0^t \int_Q \psi |u|^{\tilde{m}+p-2}| \nabla u|^2 \, dx ds 
\\
\nonumber
& \leq  N {\color{black} \left( \|V^1\|_{L_\mu(Q_t)}^p + \|V^2\|_{L_{2\mu}(Q_t)}^p\right)+ N   \| \psi^{1/p}   u\|^p_{L_{\mu'p}(Q_t)} }
\\         \label{eq:A}
& +\int_0^t  \psi'\|u\|^p_{L_p}  ds +M_t,
\end{align}
with $M_t$ is the martingale from \eqref{eq: ito product}. If $\beta  \gamma / p_{k+1} < 1$, then by virtue of Lemma \ref{lem: Revuz-Yor} and the familiar techniques of Lemma \ref{lem: right right} we obtain 
\begin{equs}
\nonumber 
&\E\left( \sup_{t \in [0,T]}\psi_t \|u_t\|^{p}_{L_{p}} + \int_0^T \int_Q \psi \left| \nabla |u|^{(\tilde{m}
+{p})/2} \right|^2 \, dx dt \right) ^{\frac{\beta \gamma} { p_{k+1} }}
\\ 
\leq N & \E \left( {\color{black} \left( \|V^1\|_{L_\mu(Q_T)}^p + \|V^2\|_{L_{2\mu}(Q_T)}^p\right) +   \| \psi^{1/p}   u\|^p_{L_{\mu'p}(Q_T)} }
 +\int_0^T  \psi'\|u\|^p_{L_p}  ds\right) ^{\frac{\beta \gamma} { p_{k+1} }},
 \\
  \label{eq:B} 
\end{equs}
where $N$ depends on $\beta, p, \varrho, \tilde{m}, T, c,K,d,r$ and $|Q|$. If $\beta \gamma / p_{k+1} \geq 1$ we have by the Burkholder-Davis-Gundy inequality
\begin{equation}          \label{eq:F}
\E\sup_{t \leq T}|M_t|^ {\beta \gamma / p_{k+1}} \leq N \E\langle M \rangle_T^{\beta \gamma / 2 p_{k+1}}.
\end{equation}
Again, as in the derivation of \eqref{eq: quadratic variation} we have 
\begin{align*}
{\color{black}\langle M\rangle_T \leq \sup_{t \leq T} \psi_t\|u_t\|_{L_p}^p \left( \|V^2\|_{L_{2\mu}(Q_T)}^p
+ \|  \psi^{1/p}u\|^p_{L_{\mu'p}(Q_T)}\right) }
\end{align*}
which combined with \eqref{eq:F} gives  by virtue of Young's inequality, for any $\varepsilon>0$
\begin{align*}
\E\sup_{t \leq T}|M_t|^ {\beta \gamma / p_{k+1}}&\leq \varepsilon \E \sup_{t \leq T} (\psi_t\|u_t\|_{L_{p}}^{p})^{\beta \gamma /  p_{k+1}}
\\
 &+ N {\color{black}\E\left( \|V^2\|_{L_{2\mu}(Q_T)}^p
+ \|  \psi^{1/p}u\|^p_{L_{\mu'p}(Q_T)}\right) ^{\beta \gamma / p_{k+1}}.}
\end{align*}
Using this and \eqref{eq:A} we get \eqref{eq:B} (for $\beta \gamma / p_{k+1} \geq 1$), provided that the quantity $\E \sup_{t \leq T} (\psi_t\|u_t\|_{L_{p}}^{{p}})^{\beta \gamma /  p_{k+1}}$ is finite, which can be achieved by a localization argument. Now we use \eqref{eq:B} with {\color{black}$p= p_k/\mu'$ }{\color{black}(notice that $p_k/\mu' \geq 2$ for $k \geq 1$) }{\color{black}and using the properties of  $\psi$ and the fact that $\bar \gamma / p_{k+1} \leq  1/ p_k$, \eqref{eq:B} yields
\begin{align}
\nonumber 
&\E\left( \sup_{t \in [\varrho_{k+1},T]} \|u_t\|^{p_k/\mu'}_{L_{p_k/\mu'}} + \int_{\varrho_{k+1}}^T \int_Q  \left| \nabla |u|^{(\tilde{m}
+{p_k/\mu'})/2} \right|^2 \, dx dt \right)^{\beta \gamma / p_{k+1} }
\\
& 
\nonumber
\leq \varrho^{-\tilde \theta } N \E \left( 1+\| u \|_{L_{p_k}((\varrho_k,T)\times Q)}^\beta +|\mathcal{V}_\mu|^\beta \right).
\end{align}}
An application of Lemma \ref{lem: embeding} (see also \eqref{eq:C} and \eqref{eq: right estimate}) gives \eqref{eq:rk}. Recall that we have assumed that $k \geq 1$. {\color{black}For $k=0$, instead of \eqref{eq:A}, we use the estimate  
\begin{align}
\nonumber
&\psi_t\| u_t\|_{L_p}^p + cp(p-1)\int_0^t \int_Q \psi |u|^{\tilde{m}+p-2}| \nabla u|^2 \, dx ds 
\\
\nonumber
& \leq  N  \left( \|V^1\|_{L_p(Q_t)}^p + \|V^2\|_{L_p(Q_t)}^p\right)+ N   \| \psi^{1/p}   u\|^p_{L_{p}(Q_t)} 
\\         \nonumber
& +\int_0^t  \psi'\|u\|^p_{L_p}  ds +M_t.
\end{align}
We apply it with $p=2$ and we proceed as above, this time raising to the power $\gamma \beta / (\tilde{m} +2\gamma)$. Following the same steps, one arrives at the estimate
\begin{align}
\nonumber 
&\E\left( \int_{\varrho_{1}}^T \int_Q  |u|^{\tilde{m}
+2 \gamma}  \, dx dt \right)^{\beta /(\tilde{m}
+2 \gamma ) }
\\
& 
\nonumber
\leq \varrho^{-\tilde \theta } N \E \left( 1+\| u \|_{L_{2}(Q_T)}^\beta +\|V^1\|^\beta_{L_2(Q_T)} +\|V^2\|^\beta_{L_2(Q_T)} \right).
\end{align}
This finishes the proof since $\tilde{m}+2 \gamma \geq \tilde{m}+ 2 \bar \gamma= p_1$.}
\end{proof}

\begin{theorem}       \label{thm: smoothing2}
Suppose that Assumptions \ref{as: nd}-\ref{as:boundednessV} are satisfied. Let $u \in \mathbb{L}_2$  be a solution of  \eqref{eq: nd quasilinear}-\eqref{eq: initial nd quasilinear} and let $\alpha >0$ and {\color{black}$\mu \in \Gamma_d$}.  Then, for all  $\rho \in (0,1)$ we have 
\begin{align}         
\label{eq:D}
\E \|u \|^\alpha _{L_\infty((\rho, T) \times  Q)} \ &\leq \rho^{-\tilde{\theta}} N \E \left( 1+\| \xi\|_{L_2(Q)}^\alpha +{\color{black}\|V^1\|_{L_\mu(Q_T)}^\alpha +\|V^2\|^\alpha_{L_{2\mu}(Q_T)}} \right)
\end{align}
where  $N$ is a constant depending only on $\alpha, \tilde{m}, T, c,K,d, \mu$, and $|Q|$, and $\tilde{\theta}>0$  is a constant depending only on $\alpha, d, \mu$ and $\tilde{m}$.
\end{theorem}
\begin{proof}
The conclusion of the theorem follows immediately
from Lemma \ref{lem: stochastic Gronwal} and Lemma \ref{thm: smoothing}.
\end{proof}

\begin{remark} 
In Theorems \ref{thm: quasilinear nd} and \ref{thm: smoothing2}, the expressions $ \| u\|_{L_\infty(Q_T)}$   and $ \| u\|_{L_\infty((\rho, T)\times Q)}$ can be replaced by $\sup_{t \in [0,T]} \|u_t\|_{L_\infty(Q)}$ and $\sup_{t \in [\rho,T]} \|u_t\|_{L_\infty(Q)}$, respectively. This follows from the fact that $u$ is a continuous $L_2(Q)$-valued process. 
\end{remark}

\begin{remark}
A cut-off argument in space, similar to the cut-off in time as it was used in the proof of Theorem \ref{thm: smoothing2}, can be used in order to derive local in space-time estimates that are applicable not only to solutions of the Dirichlet problem but to any $u$ satisfying \eqref{eq: nd quasilinear} (see, e.g.,  \cite{KOMA}).
\end{remark}

\section{Degenerate Quasilinear SPDE}
In this section, we proceed with the degenerate equation \eqref{eq: PME stratonovich}. Notice that the constant $N$ in Theorem \ref{thm: quasilinear nd} and Theorem \ref{thm: smoothing} of the previous section does not depend on the non-degeneracy constant $\theta$. Using this fact we can deduce similar estimates for the stochastic porous medium equation \eqref{eq: PME stratonovich}.  

Suppose that on $(\Omega, \mathcal{F}, \mathbb{F}, \bP)$ we are given  independent $\bR$-valued Wiener processes $\tilde{\beta}^1_t,..., \tilde{\beta}^d_t, w^1_t, w^2_t , ...$. Moreover, in this section we will assume that the domain $Q$ is convex and that the boundary $\D Q$ is of class $C^2$. The Stratonovich integral 
$$
\sum_{i=1}^d \sigma_t \D_iu_t \circ d\tilde{\beta}^i_t
$$
 in \eqref{eq: PME stratonovich} is a short notation for $$
 \frac{1}{2} \sigma_t^2 \Delta u_t \, dt + \sum_{i=1}^d \sigma_t \D_iu_t \, d \tilde{\beta}^i_t.
$$
  In the following, we consider a slightly more general class of equations. Namely, 
on $(0,T) \times Q$ we  consider the stochastic porous medium equation (SPME) of the form
\begin{equation} 
\begin{aligned}\label{eq: main equation}
du = & \left[ \Delta \left( \Phi (u) \right) +H_tu +f_t(x) \right]  \, dt + \sum_{i=1}^d \sigma_t\D_i u  \,  d\tilde{\beta}^i_t\\
+&  \sum_{k=1}^\infty\left[\nu^k_t(x) u  +g^k_t(x) \right] dw^k_t
\\
u_0= & \xi,
\end{aligned}
\end{equation}
with zero Dirichlet boundary condition on $\D Q$, 
where 
$$
H_tu:=\frac{\sigma_t^2}{2} \Delta u  +b^i_t(x)\D_iu+c_t(x)u .
$$
If we set 
\begin{align*}
\beta^k_t:=
\begin{cases}
\tilde{\beta}^k_t , & \text{for } k\in \{1,...,d\}  \\
        w^{k-d}_t, & \text{for } k \in \{ d+1,d+2,...\}
\end{cases}
\end{align*}
and 
\begin{align*}
 M^k_t(u):=
\begin{cases}
\sigma_t\D_k u  , & \text{for } k\in \{1,...,d\}\\
  \nu^{k-d}_t(x)u+g^{k-d}_t(x)      , & \text{for } k \in \{ d+1,d+2,...\},
\end{cases}
\end{align*}
we can rewrite \eqref{eq: main equation} in the more compact form 
\begin{equation}
\begin{aligned}
du = & \left[ \Delta \left( \Phi(u) \right) +H_tu +f_t(x) \right]  \, dt + \sum_{k=1}^\infty M^k_t(u)  d\beta^k_t
\\
u_0= & \xi.
\end{aligned}
\end{equation}
\begin{assumption}                   \label{as: coeff}
$$
$$
\vspace{-1 cm}
\begin{enumerate}
\item \label{item: coercivity of phi} The function $\Phi: \bR \to \bR$ is continuously differentiable,  non-decreasing, such that $\Phi(0)=0$. There exists constants $\lambda >0$, $C \geq 0$, $m \in [1, \infty)$ such that for all $r \in \bR$ we have 
$$r \Phi(r) \geq \lambda |r|^{m+1}-C, \qquad  |\Phi'(r)| \leq C|r|^{m-1}+C.
$$
\item For each $i =1,..., d$, the functions \label{as: coefficients drift}
$ b^i, c : \Omega_T \times \overline{Q} \to \bR$ are 
$\mathcal{P} \otimes \mathcal{B}(\overline{Q})$-measurable, and for all $(\omega, t) \in \Omega_T$ we have 
$ b^i \in C^2(\overline{Q};\bR)$, $c \in C^1(\overline{Q};\bR)$, and $b^i=0$ on $\D Q$. 
The functions \label{as: coefficients noise}
$\sigma: \Omega_T \to \bR$ and $\nu: \Omega_T\times \overline{Q} \to l_2$ are $\mathcal{P}$- and $\mathcal{P}\times \mathcal{B}(\overline{Q})$-measurable, respectively, and for all $(\omega,t) \in \Omega_T$ we have $\nu \in C^1(\overline{Q};l_2)$.
Moreover,  there exists a constant $K$ such that for all 
$(\omega, t,x) \in \Omega_T \times Q$ we have
\begin{align*}
&|\sigma_t| + | b^i_t(x)|+|c_t(x)|+|\nu_t(x)|_{l_2}
\\
+&  | \nabla b^i_t(x)|+|\nabla c_t(x)| + |\nabla \nu_t(x)|_{l_2}+| \nabla^2 b^i_t(x)|\leq K
\end{align*}

\item \label{as: free terms} The functions $f : \Omega_T \to H^{-1}$ and  $g^k:\Omega_T \to H^{-1}$, for $k \in \bN$, are $\mathcal{P}$-measurable and it holds that
$$
\E \int_0^T ( \|f\|_{H^{-1}}^2+ \sum_{k=1}^\infty \|g^k\|^2_{H^{-1}})  \,  dt < \infty.
$$
\item \label{as:icH-1}The initial condition $\xi$ is an $\mathcal{F}_0$-measurable $H^{-1}$-valued random variable such that $\E \|\xi\|^2_{H^{-1}}< \infty$.
\end{enumerate}
\end{assumption}

\begin{assumption}           \label{as: extra condition}
There exists a constant $\bar{c}>0$ such that  $\bar{c}|r|^{m-1} \leq \Phi'(r) $ for all $r \in \bR$.
\end{assumption}

We note that Assumption \ref{as: extra condition} implies the first part of Assumption \ref{as: coeff} \eqref{item: coercivity of phi}, that is $r \Phi(r) \geq \bar{c}|r|^{m+1}$. 

Let us set $A_t(u):= \Delta \left( \Phi(u)\right) + H_tu_t +f_t $. 
The operators are understood in the following sense:
For $u \in L_{m+1}(Q)$, we have $A_t(u) \in (L_{m+1})^*$, $M_t^k(u)\in H^{-1}$, given by 
\begin{align*}
{}_{(L_{m+1})^*}\langle A_t(u), \phi \rangle_ {L_{m+1}}:= &-\int_Q \Phi(u) \phi \, dx-  \int_Q \frac{1}{2}\sigma^2_t u  \phi   \, dx \\
&-\int_Q u \D_i (b^i_t (-\Delta)^{-1} \phi)+ \int_Q (c_tu+f_t) (-\Delta)^{-1} \phi \, dx,
\\
\left(M^k_t(u)\right)(\psi) :=& \left\{\begin{array}{lr}        -\int_Q \sigma_t u \D_k \psi \, dx , \ \ \ \  \qquad \text{for} \ k\in \{1,...,d\}\\
        \int_Q (\nu_t^{k-d}u+g^{k-d}_t) \psi \, dx \ \ \text{for } k \in \{ d+1,d+2,...\},
        \end{array}\right.
\end{align*}
for $\phi \in L_{m+1}(Q)$, $\psi \in H^1_0(Q)$, where $(-\Delta)^{-1}$ denotes the inverse of the Dirichlet Laplacian on $Q$.
Notice that for $\phi \in H^{-1}$ (in particular for $\phi \in L_{m+1}(Q)$), it holds that (see, e.g., p. 69 in \cite{MR})
$$
(M^k_t(u), \phi)_{H^{-1}}=\left(M^k_t(u)\right)((-\Delta)^{-1}\phi).
$$
\begin{definition}               \label{def: solution}
A solution of equation \eqref{eq: main equation} is a process  $u \in \mathbb{H}^{-1}_{m+1}$, such that for all $\phi \in L_{m+1}(Q)$, with probability one we have 
$$
(u_t, \phi)_{H^{-1}}= (\xi , \phi)_{H^{-1}}+\int_0^t {}_{(L_{m+1})^*}\langle A(u), \phi \rangle_{L_{m+1}} \, ds + \int_0^t (M^k(u), \phi)_{H^{-1}} d\beta^k_s,
$$
for all $t \in [0,T]$. 

\end{definition}

\subsection{Well-posedness}
In this subsection we show that the problem \eqref{eq: main equation} has  a unique solution. This will be a  consequence of  \cite[Theorems 3.6 and 3.8]{KR}, once the respective assumptions are shown to be fulfilled. This is the purpose of the following lemmata. 

\begin{remark}
In Definition \ref{def: solution}, the set of full probability on which the equality is satisfied can be chosen independently of $\phi \in L_{m+1}$. This follows by the fact that the expression
$$
\int_0^\cdot M^k(u) \, d\beta^k_s
$$
is a continuous $H^{-1}$-valued martingale, combined with
the separability of $L_{m+1}$. 
\end{remark}

\begin{lemma}
Under Assumption \ref{as: coeff} there is a constant $N$ depending only on $K$ such that for all $(\omega, t) \in \Omega_T$, $u \in L_{m+1}(Q)$ we have 
\begin{align}   \label{eq: first derivatives}
\left| \int_Q  u \D_i( b^i_t ( - \Delta )^{-1} u ) \, dx \right| \leq & N \|u \|_{H^{-1}}^2, \\  \label{eq: zero order}
\left| \int_Q  u c_t ( - \Delta )^{-1} u \, dx \right| \leq & N \|u \|_{H^{-1}}^2.
\end{align}
\end{lemma}
\begin{proof}
By continuity it suffices to show the conclusion for $u \in C^\infty_c(Q)$. We have 
\begin{align}
\int_Q  u \D_i( b^i_t ( - \Delta )^{-1} u ) \, dx& =  \int_Q  (\D_ib^i_t) u( - \Delta )^{-1} u   \, dx  +\int_Q  b^i_t u\D_i( - \Delta )^{-1}  u   \, dx.
\end{align}
Recall that $\D Q \in C^2$ which implies $(-\Delta)^{-1}u \in H^2(Q)$. Hence, writing $u= (-\Delta) (-\Delta)^{-1}u$, integration by parts gives ($(-\Delta)^{-1}u$ vanishes on $\D Q$)
\begin{align}
\nonumber
\left| \int_Q  (\D_ib^i_t) u( - \Delta )^{-1} u   \, dx \right| & \leq  \left| \int_Q (\D_{ij} b^i_t) \left(\D_j ( - \Delta )^{-1} u\right)  (-\Delta)^{-1} u \, dx \right| \\ \nonumber
& +\left| \int_Q (\D_ib^i_t)\left( \D_j( - \Delta )^{-1} u\right)^2 \, dx \right|  
\\  \label{eq: first derivatives estimate}
& \leq N \| u \|^2_{H^{-1}},
\end{align}
where we have used  Young's and Poincar\'e's inequalities. For the other term,  since $\D _j(-\Delta)^{-1}u \in H^1(Q)$ (recall that $\D Q \in C^2$) and  $b^i$ vanishes on $\D Q$, we have 
\begin{align}
\nonumber 
\int_Q  b^i_t u\D_i( - \Delta )^{-1}  u   \, dx =
&\int_Q  b^i_t \left( \D_j(-\Delta)^{-1}u \right) \D_j\D_i( - \Delta )^{-1}  u   \, dx \\  \nonumber
+&\int_Q  (\D_jb^i_t) \left( \D_j (-\Delta)^{-1}u \right) \D_i ( - \Delta )^{-1} u   \, dx.
\end{align} 
For the second term in the above equality we have by H\"older's inequality 
$$
\left| \int_Q  (\D_jb^i_t) \left(\D_j (-\Delta)^{-1}u \right)\D_i( - \Delta )^{-1}   u   \, dx \right| \leq N \|u\|^2_{H^{-1}},
$$
while for the first term we have
\begin{align}
\nonumber
\left| \int_Q  b^i_t \left( \D_j(-\Delta)^{-1}u \right) \D_i\D_j( - \Delta )^{-1}  u   \, dx \right|  &= \frac{1}{2}\left| \int_Q b^i_t \D_i \left( \D_j(-\Delta)^{-1} u \right)^2  \, dx \right| \\
\nonumber
&= \frac{1}{2}\left| \int (\D_ib^i_t)  \left(\D_j (-\Delta)^{-1} u \right)^2  \, dx \right| \\
\nonumber 
& \leq N \|u\|^2_{H^{-1}},
\end{align}
where we have used again that $b^i=0$ on $\D Q$. 
Hence, 
\begin{equation}                      \label{eq: estimate first derivatives 2}
\left| \int_Q  b^i_t u\D_i( - \Delta )^{-1}  u   \, dx\right| \leq N \|u\|_{H^{-1}}^2.
\end{equation}
Combining \eqref{eq: first derivatives estimate} with \eqref{eq: estimate first derivatives 2} we obtain  \eqref{eq: first derivatives}. Inequality \eqref{eq: zero order} follows similarly from the fact that $|c_t(x)|+|\nabla c_t(x)|\leq K$.

\end{proof}

\begin{lemma}                \label{lem: monotonicity}
Under Assumption \ref{as: coeff}, there exists a constant $N$ depending only on $K$ and $d$ such that for all $(\omega, t)\in \Omega_T$ and all  $\varphi, \psi \in L_{m+1}(Q)$ we have 
\begin{align}
\nonumber
&2 {}_{(L_{m+1})^*}\langle H_t\phi+f_t,\phi  \rangle_{L_{m+1}} +\sum_{k=1}^\infty \| M^k_t(\phi)\|_{H^{-1}}^2  
\\
\label{eq: part of coercivity}
&\leq N \left( \|\phi\|_{H^{-1}}^2+\|f_t\|^2_{H^{-1}}+  \sum_{k=1}^\infty\|g^k_t\|_{H^{-1}}^2\right),
\end{align}
and
\begin{equation}           \label{eq: monotonicity}
2{}_{(L_{m+1})^*}\langle A_t(\phi)-A_t(\psi), \phi-\psi \rangle_{L_{m+1}}+ \sum_{k=1}^\infty \| M^k_t (\phi) -M^k_t 
(\psi) \|^2_{H^{-1}} \leq N \| \phi- \psi\|^2_{H^{-1}}.
\end{equation}
\end{lemma}
\begin{proof}
We start by proving \eqref{eq: part of coercivity}.
 By  virtue of the previous lemma, it suffices to show that 
\begin{align*}
&-\sigma_t^2 \| \phi \|_{L_2}^2+ ( f_t, \phi)_{H^{-1}}+ \sum_{k=1}^\infty \| M^k_t(\phi) \|^2_{H^{-1}} 
\\
&\leq N \left(\| \phi \|^2_{H^{-1}}+\|f_t\|^2_{H^{-1}}+ \sum_{k=1}^\infty\|g^k_t\|_{H^{-1}}^2\right).
\end{align*}
Clearly it suffices to show the last inequality for $\phi \in C^\infty_c(Q)$.
To this end, we have
\begin{align*}
&-\sigma_t^2 \| \phi \|_{L_2}^2+ ( f_t, \phi)_{H^{-1}}+ \sum_{k=1}^\infty \| M^k_t(\phi) \|^2_{H^{-1}}
\\
=& -\sigma_t^2 \| \phi \|_{L_2}^2+ ( f_t, \phi)_{H^{-1}}+ \sigma_t^2 \sum_{i=1}^d \|\D_i \phi \|^2_{H^{-1}}+ \sum_{k=1}^\infty \| \nu_t^k\phi+g^k_t\|^2_{H^{-1}}
\\
\leq &  \ \sigma_t^2(\sum_{i=1}^d \|\D_i \phi \|^2_{H^{-1}}-\| \phi \|_{L_2}^2) +N \left(\| \phi \|^2_{H^{-1}}+\|f_t\|^2_{H^{-1}}+ \sum_{k=1}^\infty\|g^k_t\|_{H^{-1}}^2\right).
\end{align*}
Hence, we only have to show that 
$$
\sum_{i=1}^d \|\D_i \phi \|^2_{H^{-1}}\leq \| \phi \|_{L_2}^2.
$$
Let $\zeta \in C^\infty_c(Q)$. The action of $\D_i\phi$ on $\zeta$ is given by $-(\phi, \D_i \zeta)_{L_2}$. We have 
\begin{align*}
|(\phi, \D_i \zeta)_{L_2}|=|( (-\Delta)(-\Delta)^{-1} \phi, \D_i \zeta)_{L_2}|& =\left| \sum_{l=1}^d (\D_i  \D_l (-\Delta)^{-1} \phi, \D_l \zeta)_{L_2}\right|
\\
& \leq \|\zeta\|_{H^1_0}^2\sum_{l=1}^d \|\D_i  \D_l (-\Delta)^{-1} \phi\|_{L_2}^2.
\end{align*}
Consequently,
$$
\| \D_i \phi\|_{H^{-1}}^2 \leq \sum_{l=1}^d \|\D_i  \D_l (-\Delta)^{-1} \phi\|_{L_2}^2.
$$
It now suffices to show that 
$$
\sum_{l,i=1}^d \|\D_i  \D_l (-\Delta)^{-1} \phi\|_{L_2}^2\leq \| \phi \|^2_{L_2}.
$$
This follows from the convexity of $Q$. Namely, if $Q$ is a convex, open, bounded subset of $\bR^d$ with boundary of class $C^2$, then it holds that 
$$
\sum_{l,i=1}^d \|\D_i  \D_l v\|_{L_2(Q)}^2\leq \| \Delta v \|^2_{L_2(Q)} 
$$
for all $v \in H^2(Q) \cap H^1_0(Q)$ (see \cite[p139, Theorem 3.1.2.1, inequality (3,1,2,2)]{GRI}). Applying this to $v := (-\Delta)^{-1} \phi$ finishes the proof of \eqref{eq: part of coercivity}.

For \eqref{eq: monotonicity}, by considering \eqref{eq: part of coercivity} (with $f=0, \ g=0$), it is clear that we only have to show that 
 $$
 {}_{(L_{m+1})^*}\langle \Delta \left( \Phi (\phi) \right) - \Delta \left( \Phi (\psi) \right) , \phi -\psi \rangle_{L_{m+1}} \leq 0.
 $$
 This follows from the well-known fact (see, e.g., p.71 in \cite{MR}) that 
 $$
 {}_{(L_{m+1})^*}\langle \Delta \left( \Phi (\phi) \right) - \Delta \left( \Phi (\psi) \right) , \phi -\psi \rangle_{L_{m+1}}= -(\Phi(\phi)-\Phi(\psi), \phi-\psi)_{L_2} \leq 0,
 $$
 since $\Phi$ is non-decreasing. This completes the proof.
\end{proof}

\begin{theorem}                        \label{thm: well posedness}
Under Assumption \ref{as: coeff} there exists a unique solution of equation \eqref{eq: main equation}. 
\end{theorem}

\begin{proof}
It is straightforward to check that the operator $A$ satisfies ($A_1$) (hemi-continuity) from \cite{KR}. The fact that $A$ and $M$ satisfy ($A_2$) (monotonicity) was proved in \eqref{eq: monotonicity}. Coercivity or ($A_3$), follows from \eqref{eq: part of coercivity} combined with \eqref{item: coercivity of phi} of Assumption \ref{as: coeff}, which implies that 
$$
{}_{(L_{m+1})^*}\langle \Delta \left( \Phi(\phi)\right), \phi \rangle_{L_{m+1}} = - \left( \Phi(\phi), \phi\right)_{L_2} \leq  - \lambda \|\phi\|_{L_{m+1}}^{m+1}+C.
$$
For the growth condition ($A_4$) we have, for $v \in L_{m+1}$, 
\begin{align*}
\| \Delta \left( \Phi (v ) \right) + H_tv+f_t\|_{(L_{m+1})^*} & \leq \|\Phi(v)\|_{L_{(m+1)/m}} +  \|H_tv+f_t\|_{(L_{m+1})^*}
\\
& \leq N+ N\|v\|^{m}_{L_{m+1}}+\|H_tv+f_t\|_{(L_{m+1})^*},
\end{align*}
where we have used Assumption \ref{as: coeff}, \eqref{item: coercivity of phi}.
Then, notice  that for $\phi \in C^\infty_c(Q)$
\begin{align*}
&{}_{(L_{m+1})^*}\langle H_tv+f_t, \phi \rangle_{L_{m+1}} \\
=&-  \int_Q \frac{1}{2}\sigma^2_t u  \phi   \, dx -\int_Q u \D_i (b^i_t (-\Delta)^{-1} \phi)\, dx+ \int_Q (c_tu+f_t) (-\Delta)^{-1} \phi \, dx
\\
\leq &N\|u\|_{L_{(m+1)/m}} \|\phi\|_{L_{m+1}}+N \|u\|_{L_2}\|\phi\|_{H^{-1}}+ N\|f\|_{H^{-1}}\|\phi\|_{H^{-1}}\\
\leq & N \left( 1+ \|f_t \|_{H^{-1}}^{2m/(m+1)}+ \|v\|^{m
}_{L_{m+1}} \right) \|\phi\|_{L_{m+1}}.
\end{align*}
Hence, 
$$
\| \Delta \left( \Phi (v ) \right) + H_tv+f_t\|_{(L_{m+1})^*}  \leq N \left( 1+ \|f_t \|_{H^{-1}}^{2m/{m+1}}+ \|v\|^{m}_{L_{m+1}} \right),
$$
where $N$ depends only on $m, K, C, d$ and $|Q|$.  This finishes the verification of the assumptions of \cite[Theorems 3.6 and 3.8]{KR}, an application of which concludes the proof.
\end{proof}

\subsection{Regularity}
In this section we add a viscosity term of magnitude $\varepsilon$ to equation
\eqref{eq: main equation} and  show that the corresponding equation and its solution $u^\varepsilon$ satisfy the assumptions of Theorem \ref{thm: quasilinear nd}, which yields supremum estimates for $u^\varepsilon$ uniformly in $\varepsilon$. Then, we show that $u^\varepsilon$ converges to $u$ and  pass to the limit to obtain the desired estimates for the $u$. 

First, we consider an approximating equation where the non-linear term is Lipschitz continuous, that is, for $\varepsilon>0$, on $Q_T$ 
\begin{equation}                     \label{eq: double approximation}
\begin{aligned}
d u_t&= \left[\Delta \left( \bar{\Phi}(u_t)\right)  + \varepsilon\Delta u_t+Hu_t+f_t\right]dt+ M^k_t(u_t) \,  d\beta^k_t \\
u_0&=\xi,
\end{aligned}
\end{equation}
with zero Dirichlet boundary conditions on $\D Q$.
Let us set 
$$
\mathcal{K}_p := \E\|\xi\|_{L_p}^p +\E \int_0^T \left(\|f\|_{L_p}^p+\| |g|_{l_2}\|^p_{L_p}\right)  \, dt.
$$
As in \cite[Lemma B.1]{BenRoc} we have the following.
\begin{lemma}                                \label{lem: Lipschitz non-linearity}
Assume that Assumption \ref{as: coeff}  \eqref{as: coefficients drift}, \eqref{as: free terms}, \eqref{as:icH-1} is satisfied and  that $\bar{\Phi}: \bR \to \bR$ is a Lipschitz continuous, non-decreasing function with $\bar{\Phi}(0)=0$. Then, equation \eqref{eq: double approximation} has a unique solution $u$ in $\mathbb{H}^{-1}_2$. Moreover, if $K_2< \infty$, then $u \in \mathbb{L}_2$ and for any $p \geq 2$ the following estimate holds
\begin{equation}                             \label{eq: p estimates}
\E \sup_{t \leq T} \|u_t\|^p_{L_p} +  \E \int_0^T \||u|^{(p-2)/2}|\nabla u| \|_{L_2}^2 \, dt \leq  N  \mathcal{K}_p,
\end{equation}
where 
 $N$ is a constant depending only on $\varepsilon, K, d, T$ and $p$.

 \end{lemma}
\begin{proof}
The existence and uniqueness of solutions in $\mathbb{H}^{-1}_2$ follows from Theorem \ref{thm: well posedness}. Therefore, we only have to show that $u \in \mathbb{L}_2$ and \eqref{eq: p estimates} under the assumption that $\mathcal{K}_2 < \infty$. Let $(e_i)_{i=1}^\infty$ be an orthonormal basis of $H^{-1}$
consisting of eigenvectors of $-\Delta$ and let $\Pi_n : H^{-1} \to \text{span}\{ e_1, ..., e_n\}$ be the orthogonal projection onto the span of the first $n$ eigenvectors.
Consider the Galerkin approximation
\begin{equation}                     \label{eq: Galerkin approximation}
\begin{aligned}
d u^n_t&= \Pi_n \left[ \Delta \left( \bar{\Phi}(u^n_t)\right)  + \varepsilon\Delta u^n_t+Hu^n_t+f_t\right] dt+ \Pi_n  M^k_t(u^n_t) \,  d\beta^k_t \\
u^n_0&=\Pi_n \xi.
\end{aligned}
\end{equation}
Under the assumptions of the lemma, it is very well known from the theory of stochastic evolution equations (see \cite{KR}) that the Galerkin scheme above  has a unique solution $u^n$ which converges 
 weakly in $L_2( \Omega_T, \mathcal{P}; L_2(Q))$ to $u$ (in fact, this is how the solution $u$ is constructed). Notice that the restriction of $\Pi_n$ to $L_2$ is again the orthogonal projection (in $L_2$) onto $\text{span}\{ e_1, ..., e_n\}$. Consequently, for any $\phi, \psi \in C_c^\infty$ we have $-(\phi, \Pi_n \D_i \psi)_{L_2}= (\D_i \Pi_n \phi,  \psi)_{L_2}$ which remains true for $\phi ,\psi \in L_2$. Hence, by It\^o's formula, we have 
\begin{align*}
\|u^n_t\|_{L_2}^2 \leq  \|\xi\|_{L_2}^2& - \int_0^t 2\left( \D_j u^n,  \D_j  \left( \bar{\Phi}(u^n) \right)  + \varepsilon \D_j u^n+ \frac{\sigma^2}{2}\D_j u^n \right)_{L_2} \, ds
\\
& +\int_0^t  2 \left( b^i \D_i u^n+c u^n+f, u^n\right)_{L_2} \, ds
\\
&+ \int_0^t\left( \sum_{i=1}^d \sigma^2  \| \D_iu^n\|_{L_2}^2+  \sum_{k=1}^\infty \|\nu^ku^n+g^k\|^2_{L_2}\right) \, ds
\\
&+ \int_0^t 2 \left( M^k u^n+g^k, u^n\right) \, d \beta^k_s.
\end{align*}
Since $ \bar{ \Phi}$ is a non-decreasing Lipschitz continuous function, we have 
$$
\left( \D_j u^n,  \D_j \left(  \bar{\Phi}(u^n) \right) \right)_{L_2} \geq 0.
$$
It then follows by standard arguments (see, e.g., the proof of Theorem 4 in \cite{ROZ}) that 
$$
\E\int_0^T \| u^n\|^2_{H^1_0} \, dt \leq N \mathcal{K}_2, 
$$
where $N$ depends only on $\varepsilon, K, d,$ and  $T$. Since the Galerkin approximation $u^n$ converges weakly in $L_2( \Omega_T, \mathcal{P}; L_2(Q))$ to $u$, taking $\liminf$ as $n \to \infty$ in the above inequality  gives
$$
\E\int_0^T \| u\|^2_{H^1_0} \, dt \leq N \mathcal{K}_2.
$$
Moreover, since $u \in L_2(\Omega_T, \mathcal{P}; H^1_0(Q))$ and satisfies \eqref{eq: double approximation}, it follows (see \cite[Theorem 2.16]{KR}) that is has a continuous $L_2$-valued version which implies that $u \in \mathbb{L}_2$.  From here, one can deduce \eqref{eq: p estimates} by following step by step the proof of Lemma 2 from \cite{KOMA}, keeping in mind that $\bar{\Phi}'(u) \geq 0$. 

\end{proof} 
We use the previous result to obtain the required regularity for the solution of the SPME in the presence of a non-degenerate viscosity term,
\begin{equation}                \label{eq: single approximation}
\begin{aligned}
du_t& =\left[ \Delta \left( \Phi (u) \right) +\varepsilon \Delta u_t +H_tu_t +f_t\right] \, dt +  M^k_t (u_t)   d\beta^k_t \\
u_0&= \xi.
\end{aligned}
\end{equation}
\begin{lemma}                \label{lem: regularity}
Suppose that Assumption \ref{as: coeff} holds. Then, there exists a unique $\mathbb{H}^{-1}_{m+1}$-solution of equation \eqref{eq: single approximation}. If $\xi \in L_{m+1}(\Omega; L_{m+1}(Q))$,  $ f \in L_{m+1}(\Omega_T; L_{m+1}(Q))$, and $g \in L_{m+1}(\Omega_T; L_{m+1}(Q; l_2))$, then we have that $u, \Phi(u) \in L_2(\Omega_T; H^1_0)$ and $\nabla \Phi (u) = \Phi'(u) \nabla u$. In particular, $u \in \mathbb{L}_2$. 
\end{lemma}
\begin{proof}
The fact that \eqref{eq: single approximation} has a unique $\mathbb{H}^{-1}_{m+1}$-solution follows from Theorem \ref{thm: well posedness}. For the remaining properties, let us consider the approximation
\begin{equation}        \label{eq: approximation with Lipschitz}
\begin{aligned}
du^n_t& =\left[ \Delta \left( \Phi _n(u^n_t)\right)+\varepsilon \Delta u^n_t +H_tu^n_t +f_t\right] \, dt +   M^k_t (u^n_t) \,  d\beta^k_t \\
u_0&= \xi,
\end{aligned}
\end{equation}
where for $n \in \mathbb{N}$, $\Phi_n: \bR \to \bR$ is defined by 
$$
\Phi_n(r)= \int_0^r \min\{\Phi'(s), n\} \, ds.
$$
 By Lemma \ref{lem: Lipschitz non-linearity}, equation \eqref{eq: approximation with Lipschitz} has a unique solution $u^n$ in $\mathbb{H}^{-1}_2$ which moreover belongs to $\mathbb{L}_2$, and for all 
 $q \in [2, m+1]$ we have 
\begin{equation}            \label{eq: stability in Lp}
\begin{aligned}
&\E \sup_{t \leq T} \| u^n \|^q_{L_q}+ \E  \int_0^T\int_Q  | u^n |^{q-2}|\nabla u^n |^2 \, dx dt  \\
&\leq N  \left( \E \|\xi\|^q_{L_q} + \E \int_0^T \| f\|^q_{L_q} + \| |g|_{l_2} \|_{L_q}^q \, dt  \right)< \infty,
\end{aligned}
\end{equation}
where $N$ depends only on $K, T, d, q$ and $\varepsilon$. 

 Let $\Psi_n(r) = \int_0^r \Phi_n(s) \, ds$. By Theorem 3.1 in \cite{KRITO} we have 
\begin{align}                         \nonumber
\int_Q \Psi_n (u^n_t) \,  dx &= \int_Q \Psi_n (\xi) \,  dx- \int_0^t \int_Q | \nabla \Phi_n(u^n) |^2 + \varepsilon\Phi'(u^n) |\nabla u^n|^2 \, dx ds \\   \nonumber
&+ \int_0^t (  b^i\D_iu^n+cu^n+f, \Phi_n (u^n) )_{L_2} \, ds   \\   
 \nonumber
 &+\int_0^t \frac{1}{2}(\Phi'_n(u^n),| \nu u^n+g|^2_{l_2}) _{L_2}\, ds
 \\
  \label{eq: Ito psi_n}
 & + \int_0^t ( M^k(u^n) , \Phi_n(u^n) )_{L_2} \, d\beta^k_s.
\end{align}
Notice that by Assumption \ref{as: coeff}   we have
\begin{align*}
 -\Phi'_n(u^n) |\nabla u^n|^2   \leq &  0,
\\
 (cu^n+f, \Phi_n(u^n) )_{L_2}  \leq & N ( \|u^n\|^{m+1}_{L_{m+1}}+\|f\|^{m+1}_{L_{m+1}} )
\\
( b^i\D_iu^n, \Phi_n(u^n))_{L_2}=&-((\D_ib^i)u^n, \Phi_n(u^n))_{L_2}- ( b^iu^n, \D_i\Phi_n(u^n))_{L_2}
\\
&\leq N(1+\|u^n\|_{L_{m+1}}^{m+1})+\frac{1}{2}\|\nabla \Phi_n(u^n)\|_{L_2}^2
 \end{align*}
 and 
\begin{align*}
(\Phi'_n(u^n),| \nu u^n+g|^2_{l_2})_{L_2}
& \leq N ( 1+\|u^n\|^{m+1}_{L_{m+1}}+\||g|_{l_2}\|^{m+1}_{L_{m+1}}),
 \end{align*}
 for a constant $N$ depending only on $C,K,d,$ and $|Q|$. Hence, after a localization argument we obtain for all $t \in [0,T]$
 \begin{align}
\nonumber               
&\E\int_Q \Psi_n (u^n_t) \,  dx +\E \int_0^t \| \nabla \Phi_n(u^n) \|_{L_2}^2\, dt 
\\
\label{eq:RemarkLater}
&\leq N \E \left(1+  \|\xi\|^{m+1}_{L_{m+1}} +  \int_0^t \| f\|^{m+1}_{L_{m+1}} + \| |g|_{l_2} \|_{L_{m+1}}^{m+1}  +\|u^n\|^{m+1}_{L_{m+1}}\, dt \right),
\end{align}
for a constant $N$ depending only on $C,K,d,$ and $|Q|$,
which in particular gives by \eqref{eq: stability in Lp}
\begin{align}
\nonumber               
&\E \int_0^T \| \nabla \Phi_n(u^n) \|_{L_2}^2\, dt 
\\
\label{eq: stability phi_n}
&\leq N \E \left(1+  \|\xi\|^{m+1}_{L_{m+1}} +  \int_0^T \| f\|^{m+1}_{L_{m+1}} + \| |g|_{l_2} \|_{L_{m+1}}^{m+1}  \, dt \right)< \infty,
\end{align}
where $N$ depends only on $K, T,d ,m,C, |Q|$ and $\varepsilon$.  By \eqref{eq: stability in Lp} and \eqref{eq: stability phi_n} we have for a (non-relabeled) subsequence   
\begin{equation}                        \label{eq: weak converg Lp}
\begin{aligned}                    
u^n &\rightharpoonup v  \ \text{in}  \ L_2(\Omega_T ; H^1_0(Q)), \\
u^n &\rightharpoonup v  \ \text{in}  \ L_{m+1}(\Omega_T ; L_{m+1}(Q)), \\
\Phi_n(u_n) &\rightharpoonup \eta  \ \text{in}  \ L_2(\Omega_T ; H^1_0(Q)),   \\
u^n_T  &\rightharpoonup u^ \infty\ \  \text{in}  \ L_{m+1}(\Omega; L_{m+1}(Q)),
\end{aligned}
\end{equation}
for some $v$, $\eta$ and $u^ \infty$.  Recall that we want to show that $u, \Phi(u) \in L_2(\Omega_T; H^1_0)$. For this, we will show that $u=v$ and $\Phi(v)= \eta$ by using standard techniques from the theory of monotone operators (see, e.g., \cite{KR}).  Notice that by \eqref{eq: weak converg Lp} we also have 
\begin{equation}                        \label{eq: weak converg H -1}
\begin{aligned}                    
u^n &\rightharpoonup v  \ \text{in}  \ L_2(\Omega_T ; H^{-1}), \\
\Delta \Phi_n(u_n) &\rightharpoonup \Delta \eta  \ \text{in}  \ L_2(\Omega_T ; H^{-1}),   \\
u^n_T  &\rightharpoonup u^ \infty\ \  \text{in}  \ L_2(\Omega; H^{-1}).
\end{aligned}
\end{equation}
As in Section 3.5 in \cite{KR}, by passing to the weak limit in \eqref{eq: approximation with Lipschitz} we have in $H^{-1}$ for almost all $(\omega, t)$
\begin{equation}                \label{eq: equality after limit}
v_t =\xi+\int_0^t \left( \Delta \eta +\varepsilon \Delta v +Hv +f\right) \, ds +  \int_0^t  M^k (v) \, d\beta^k_s,
\end{equation}
and, almost surely,
\begin{equation}
u^\infty =\xi+ \int_0^T \left( \Delta \eta +\varepsilon \Delta v +Hv +f\right) \, ds +  \int_0^T  M^k (v) \, d\beta^k_s.
\end{equation}
Hence, we can choose a version of $v$ that is  a continuous,  adapted, $H^{-1}$-valued process. It follows that \eqref{eq: equality after limit} holds  for all $t \in [0,T]$ on a set of probability one  and that almost surely $v_T= u^\infty$. To ease the notation let us set 
\begin{align*}
A^n_t(\varphi)&:=\Delta \left( \Phi _n(\varphi)\right)+\varepsilon \Delta \varphi +H_t \varphi  +f_t
\\
A_t(\varphi)&:=\Delta \left( \Phi (\varphi)\right) +\varepsilon \Delta \varphi +H_t \varphi  +f_t
\end{align*}
for $\varphi \in L_{m+1}(Q)$. Let $y$ be a predictable $L_{m+1}(Q)$-valued process, such that 
$$
\E \int_0^T \| y\|^{m+1}_{L_{m+1}}  \, dt < \infty.
$$
For $c>0$ we set 
\begin{align}
\nonumber
\mathcal{O}_n:= & \E \int_0^T e^{-ct}2 {}_{(L_{m+1})^*}\langle A^n(u^n)-A(y), u^n-y \rangle_{L_{m+1}} \, dt 
\\
\nonumber
 +& \E \int_0^T e^{-ct}\sum_{k=1}^\infty\|M^k (u^n)- M^k( y) \|^2_{H^{-1}} \, dt -\E \int_0^T c e^{-ct}\|u^n - y\|^2_{H^{-1}}  \, dt.
\end{align}
Notice that due to Lemma \ref{lem: monotonicity} we have for $c>0$ large enough (independent of $n$)
\begin{equation*}
\begin{aligned}
\mathcal{O}_n&= \E \int_0^T 2 e^{-ct} {}_{(L_{m+1})^*}\langle A^n(u^n)-A^n(y), u^n-y \rangle_{L_{m+1}} \, dt
\\
& +\E \int_0^T  e^{-ct} \sum_{k=1}^\infty\|M^k (u^n)- M^k (y) \|^2_{H^{-1}}  \, dt \\
& - \E \int_0^T c e^{-ct}\|u^n - y\|^2_{H^{-1}}  \, dt \\
&+\E \int_0^T e^{-ct} 2{}_{(L_{m+1})^*}\langle A^n(y)-A(y) , u^n-y \rangle_{L_{m+1}} \, dt 
\\
& \leq \E \int_0^T  e^{-ct}2{}_{(L_{m+1})^*}\langle A^n(y)-A(y) , u^n-y \rangle_{L_{m+1}} \, dt.
\end{aligned}
\end{equation*}
Moreover, one can easily see that by the properties of $\Phi_n$ we have that $ \Phi_n (y) \to \Phi(y)$ strongly in $L_{m+1}(\Omega_T; L_{m+1}(Q))$, which combined with \eqref{eq: weak converg Lp} gives 
$$
\lim_{n \to \infty} \E \int_0^T e^{-ct} 2{}_{(L_{m+1})^*}\langle A^n(y)-A(y) , u^n-y \rangle_{L_{m+1}} \, dt =0.
$$
Consequently, we have 
\begin{equation}                     \label{eq: limsup On}
\limsup_{ n \to  \infty} \mathcal{O}_n \leq 0.
\end{equation}
We also set 
\begin{equation*}
\mathcal{O}^1_n:= \E \int_0^T e^{-ct} \left( 2 {}_{(L_{m+1})^*}\langle A^n(u^n), u^n \rangle_{L_{m+1}} +\sum_{k=1}^\infty \|M^k (u^n) \|^2_{H^{-1}} - c \| u^n\|^2_{H^{-1}}\right)  \, dt 
\end{equation*}
and $\mathcal{O}_n^2= \mathcal{O}_n-\mathcal{O}_n^1$. By It\^o's formula (see \cite[Theorem 2.17]{KR}) we have 
\begin{align*}
&e^{-cT}\|u^n_T\|_{H^{-1}}^2-\|\xi\|_{H^{-1}}^2 
\\
&= \int_0^T e^{-ct}
\left( 2 {}_{(L_{m+1})^*}\langle A^n(u^n), u^n \rangle_{L_{m+1}} +\sum_{k=1}^\infty \|M^k (u^n) \|^2_{H^{-1}} \right)\, dt \\
&-\int_0^T e^{-ct}c \| u^n\|^2_{H^{-1}} \, dt + \int_0^T e^{-ct} (M^k(u^n), u^n)_{H^{-1}} \, d\beta^k_t.
\end{align*}
By the estimates in \eqref{eq: stability in Lp} one can easily see that 
$$
\E \left( \int_0^T \sum_{k=1}^\infty(M^k(u^n),u^n)_{H^{-1}}^2 \, dt \right)^{1/2}< \infty,
$$
which implies that the expectation of the last term at the right hand side of the above equality vanishes. Hence,
\begin{equation*}
\mathcal{O}_n^1= \E e^{-cT}\|u^n_T\|_{H^{-1}}^2-\E \|\xi\|_{H^{-1}}^2,
\end{equation*}
from which we get that 
\begin{equation}
\limsup_{n \to \infty} \mathcal{O}_n^1= \E e^{-cT}\|v_T\|_{H^{-1}}^2-\E \|\xi\|_{H^{-1}}^2+ \delta e^{-cT}
\end{equation}
with $\delta:= \limsup_{n \to \infty }\E\|u^n_T\|_{H^{-1}}^2-\E\|v_T\|_{H^{-1}}^2 \geq 0$, due to \eqref{eq: weak converg H -1}. On the other hand, by \eqref{eq: weak converg Lp} and \eqref{eq: stability in Lp} 
it follows that the quantity 
$$
\E \esssup_{t \in [0,T]} \|v_t\|^2_{L_2}+ \E \int_0^T\|\nabla v \|^2_{L_2} \, dt
$$
can be estimated by the right hand side of \eqref{eq: stability in Lp} with $q=2$. In particular, this implies that 
$$
\E \left( \int_0^T \sum_{k=1}^\infty(M^k(v),v)_{H^{-1}}^2 \, dt \right)^{1/2}< \infty.
$$
Hence, by \eqref{eq: equality after limit} and It\^o's formula we obtain
\begin{align*}
\E e^{-cT}\|v_T\|_{H^{-1}}^2 &= \E \|\xi\|_{H^{-1}}^2 
\\
&+ \E \int_0^T e^{-ct}
2 {}_{(L_{m+1})^*}\langle \Delta \eta +\varepsilon \Delta v +Hv +f, v \rangle_{L_{m+1}} \, dt 
\\
&+\E \int_0^T e^{-ct}\sum_{k=1}^\infty \|M^k(v) \|^2_{H^{-1}} \, dt 
-\E \int_0^T e^{-ct}c \| v\|^2_{H^{-1}} \, dt.
\end{align*}
Hence,
\begin{align}
\nonumber
\limsup_{n \to \infty} \mathcal{O}_n^1&=\E \int_0^T e^{-ct}
2 {}_{(L_{m+1})^*}\langle \Delta \eta +\varepsilon \Delta v +Hv +f, v \rangle_{L_{m+1}} \, dt 
\\
&+\E \int_0^T e^{-ct}\sum_{k=1}^\infty \|M^k (v) \|^2_{H^{-1}} \, dt 
-\E \int_0^T e^{-ct}c \| v\|^2_{H^{-1}} \, dt+ \delta e^{-cT}.
\end{align}
Moreover, by \eqref{eq: weak converg Lp} we have 
\begin{align}
\nonumber
\lim_{n \to \infty} \mathcal{O}^2_n &=\E \int_0^T e^{-ct}
2 \left( {}_{(L_{m+1})^*}\langle A(y), y \rangle_{L_{m+1}}-{}_{(L_{m+1})^*}\langle A(y), v \rangle_{L_{m+1}} \right)  \, dt \\
\nonumber
&-\E \int_0^T e^{-ct}
2 {}_{(L_{m+1})^*}\langle \Delta \eta +\varepsilon\Delta v +Hv+f, y \rangle_{L_{m+1}} \, dt
\\
\nonumber
&+\E \int_0^T e^{-ct}\left(\sum_{k=1}^\infty  \|M^k (y) \|^2_{H^{-1}}- 2( M^k (y), M^k(v))_{H^{-1}} \right)\, dt 
\\
&+\E \int_0^T e^{-ct}c \left(2(v, y)_{H^{-1}}- \| y\|^2_{H^{-1}}\right) \, dt.
\end{align}
Consequently,
\begin{align}
\nonumber
&\E \int_0^T e^{-ct}2 {}_{(L_{m+1})^*}\langle \Delta \eta + \varepsilon \Delta v +Hv +f-A(y), v-y \rangle_{L_{m+1}} \, dt \\
\nonumber
 + &\E \int_0^T e^{-ct}2\sum_{k=1}^\infty\|M^k (v)- M^k (y) \|^2_{H^{-1}}  \, dt  \\
 \nonumber
  - &\E \int_0^T c e^{-ct}\|v- y\|^2_{H^{-1}}  \, dt + \delta e^{-cT} = \limsup_{n \to \infty} \mathcal{O}^1_n + \lim_{n \to \infty} \mathcal{O}^2_n
  \\                                 \label{eq: testing}
   =&  \limsup_{n \to \infty} \mathcal{O}_n \leq 0,
\end{align}
by \eqref{eq: limsup On}. By choosing $y=v$ in \eqref{eq: testing} we obtain that $\delta=0$. Moreover, it follows that 
\begin{align*}
\nonumber
&\E \int_0^T e^{-ct}2 {}_{(L_{m+1})^*}\langle \Delta \eta + \varepsilon \Delta v +Hv +f -A(y), v-y \rangle_{L_{m+1}} \, dt \\
- &\E \int_0^T c e^{-ct}\|v - y\|^2_{H^{-1}}  \,  dt \leq 0.
\end{align*}
Let $z$ be a predictable process with values in $L_{m+1}(Q)$ with $\E \int_0^T \|z\|_{L_{m+1}}^{m+1} \, dt < \infty$ and choose in the above inequality $y= v- \lambda z$ for $\lambda > 0$. Then, we have
\begin{align*}
\nonumber
&\E \int_0^T e^{-ct}2 \lambda {}_{(L_{m+1})^*}\langle \Delta \eta + \varepsilon \Delta v +Hv +f -A(v - \lambda z), z \rangle_{L_{m+1}} \, dt \\
- &\E \int_0^T c \lambda^2 e^{-ct}\|z\|^2_{H^{-1}}  \,  dt \leq 0.
\end{align*}
Dividing by $\lambda$, letting $\lambda \to 0$ and using the hemi-continuity property we obtain 
$$
\E \int_0^T {}_{(L_{m+1})^*}\langle \Delta \eta - \Delta \left( \Phi(v)\right) , z \rangle_{L_{m+1}} \, dt \leq 0.
$$
Since $z$ was arbitrary, we have $\Delta \eta = \Delta \left( \Phi (v) \right) $. This shows that $v$ is a solution of \eqref{eq: single approximation}, and by uniqueness, we have $u=v$, $\Phi(u)= \Phi(v) = \eta$. This finishes the proof.
\end{proof}
\begin{remark} \label{rem:H10}
Suppose that Assumption \ref{as: extra condition} also holds and let $u^\varepsilon$ be the solution of  \eqref{eq: single approximation}. By writing It\^o's formula for $\|\Psi(u^\varepsilon_t)\|_{L_1(Q)}$, where $\Psi(r)= \int_0^r \Phi(s) \, ds$, similarly to  \eqref{eq:RemarkLater} and after applying Gronwal's  lemma one has
$$
\E \int_0^T \|\nabla \Phi(u^{\varepsilon})\|_{L_2}^2 \, dt \leq N \E \left(1+  \|\xi\|^{m+1}_{L_{m+1}} +  \int_0^T \| f\|^{m+1}_{L_{m+1}} + \| |g|_{l_2} \|_{L_{m+1}}^{m+1}  \, dt \right),
$$
 where $N$ is independent of $\varepsilon$.
\end{remark}

We will also need the following.

\begin{lemma}           \label{lem: lsc}
Let $u^n$ be real-valued functions on $Q$  such that $u^n \rightharpoonup u$ in $H^{-1}$ for some $u \in H^{-1}$. Then for any $p \in [1, \infty]$ 
$$
\|u\|_{L_p} \leq \liminf_{n \to \infty} \|u^n\|_{L_p}.
$$
\end{lemma}

\begin{proof}
Suppose first that $p\in [1, \infty)$. We assume that $\liminf_{n \to \infty} \|u^n\|_{L_p}< \infty$ or else there is nothing to prove. Under this assumption there exists a subsequence $u^{n_k}$ with $\lim_k \|u^{n_k}\|_{L_p}^p= \liminf_n  \|u^{n}\|_{L_p}^p$ and $v \in L_p(Q)$ such that $u^{n_k} \rightharpoonup v$ in $L_p(Q)$. It follows that $u=v \in L_p(Q)$, which finishes the proof since $\|v\|_{L_p} \leq \liminf \|u^{n_k}\|_{L_p}^p$.  For $p= \infty$ we have the following. We know that the conclusion holds for all $p \in [1, \infty)$. Hence,
$$
\|u\|_{L_p} \leq |Q|^{1/p}\liminf_{n \to \infty} \|u^n\|_{L_\infty},
$$
The assertion follows by letting $p \to \infty$. 
\end{proof}

We can now present our main theorems.

\begin{theorem}  \label{thm: theorem boundedness}
Suppose that  Assumptions \ref{as: coeff} and \ref{as: extra condition} are satisfied with $m>1$. Let {\color{black}$\mu \in \Gamma_d$} and  let $u \in \mathbb{H}^{-1}_{m+1}$ be the unique solution of \eqref{eq: main equation}. Then, we have 
\begin{equation}          
\E \|u \|^2 _{L_\infty(Q_T)} \ \leq N \E \left( \|\xi \|_{L_\infty(Q)}^2 +{\color{black}|S_\mu(f,g)|^2} \right),
\end{equation}
where 
$$
{\color{black}\mathcal{S}_r(f,g)=1+\|f\|_{L_\mu(Q_T)}+ \| |g|_{l_2} \|_{L_{2\mu}(Q_T)},}
$$
and $N$ is a constant depending only on $ m, T, \bar{c},K, d, \mu$ and $|Q|$. 
\end{theorem}
\begin{proof}
\emph{Step 1:} In a first step we assume that $|\xi(x)|$, $|f_t(x)|$, $|g_t(x)|_{l_2}$ are bounded uniformly in $(\omega,t,x)$. 
Let $u^\varepsilon$ denote the unique solution of the problem \eqref{eq: single approximation}. 
By Assumptions \ref{as: coeff} and \ref{as: extra condition} we have that equation \eqref{eq: single approximation} satisfies Assumptions \ref{as: nd}-\ref{as:boundednessV} with  $\theta= \varepsilon$ , $c= \bar{c}$, $\tilde{m}=m-1$, and 

$$
a^{ij}_t(x,r)= \Phi'(r)I_{i=j}+\varepsilon I_{i=j}r+\sigma^2 _tr/2, $$
\begin{equs}
F^i_t(x,r)&=b^i_t(x)r,\qquad\qquad \qquad   F_t(x,r)&&= (c_t(x)-\D_ib^i_t(x))r+f_t(x),
\\
 g^{ik}_t(x,r)&=I_{i=k, k \leq d}\sigma_t r,\qquad \qquad G^k_t(x,r)&&= I_{k>d} (\nu ^{k-d}_t(x)r + g^{k-d}_t(x)),
\\
V^1_t(x)&= |f_t(x)|, \qquad \qquad \qquad V^2_t(x)&&= |g_t(x)|_{l_2}.
\end{equs}
By Lemma \ref{lem: regularity} we have that $u^\varepsilon$ satisfies equation \eqref{eq: single approximation} also in the sense of Definition \ref{def: definition L2}.
Therefore, by Theorem \ref{thm: quasilinear nd} we have 
\begin{equation}          \label{eq: boundedness epsilon}
\E \|u^\varepsilon \|^2 _{L_\infty(Q_T)} \ \leq N \E \left( 1+\|\xi \|_{L_\infty(Q)}^2 +|S_\mu(f,g)|^2 \right),
\end{equation}
where $N$ is a constant depending only on $ m, T, \bar{c},K,d, r$, and $|Q|$. By It\^o's formula, the monotonicity of $\Phi$, and \eqref{eq: part of coercivity}, one can easily see that 
for a constant $N$ independent of $\varepsilon$ we have 
$$
\|u_t-u_t^\varepsilon\|_{H^{-1}}^2 \leq N \int_0^t \|u-u^\varepsilon\|_{H^{-1}}^2 + \varepsilon2 {}_{(L_{m+1})^*}\langle \Delta u^\varepsilon , u^\varepsilon-u  \rangle_{L_{m+1}} \, ds +M^\varepsilon_t,
$$
for a local martingale $M^\varepsilon_t$.
Hence, 
\begin{equation}                           \label{eq: viscosity convergence}
\E  \|u^\varepsilon_t-u_t\|^2_{H^{-1}} \leq N \E \int_0^T \varepsilon | {}_{(L_{m+1})^*}\langle \Delta u^\varepsilon , u^\varepsilon-u  \rangle_{L_{m+1}}| \,  dt.
\end{equation}

Moreover, by It\^o's formula,  Assumption \ref{as: coeff}, \eqref{item: coercivity of phi}, and \eqref{eq: part of coercivity}
we have for a constant $N$ independent of $\varepsilon$
$$
\E \int_0^T \| u^\varepsilon \|_{L_{m+1}}^{m+1} \, dt \leq N \E \left( 1+ \| \xi\|_{H^{-1}}^2+ \int_0^T \|f\|^2_{H^{-1}}+ \sum_{k=1}^\infty\|g^k\|_{H^{-1}}^2 \, dt \right)< \infty.
$$
The same estimate holds for $u$. Hence, by virtue of \eqref{eq: viscosity convergence}, we have
\begin{equation*}         \label{eq: u epsilon to u}
\lim_{ \varepsilon \to 0} \E  \int_0^T  \| u_t-u_t^\varepsilon\|_{H^{-1}}^2\, dt=0.
\end{equation*}
In particular, for a sequence $\varepsilon_k \to 0$ we have $\|u^{\varepsilon_k}_t-u_t\|_{H^{-1}} \to 0$ for almost all $(\omega,t)$.
By Lemma \ref{lem: lsc}, Fatou's lemma, and \eqref{eq: boundedness epsilon}, we  have for any $p \in[1, \infty)$
\begin{align*}
\E \|u \|^2 _{L_p(Q_T)}  & \leq \liminf_k \E \|u^{\varepsilon_k} \|^2_{L_p(Q_T)}
\leq \liminf_k N \E \|u^{\varepsilon_k} \|^2 _{L_\infty(Q_T)}
\\
&\leq  N \E \left( 1+\|\xi \|_{L_\infty(Q)}^2 +|S_\mu(f,g)|^2 \right),
\end{align*}
with $N$ independent of $p$, and the result follows by letting $p \to \infty$.

\emph{Step 2:} For general $\xi,  f , g$ we set 
\begin{align*}
\xi^n&= ((-n)\vee \xi) \wedge n, \qquad 
f^n= ((-n)\vee f) \wedge n, 
\\
g^n&= \sum_{k=1}^{n}\left(  ((-C_n)\vee g^{k})\wedge C_n   \right) e_k 
\end{align*}
where $(e_k)_{k=1}^\infty$ is the usual orthonormal basis of $l_2$ and $C_n \geq 0$  are chosen such that 
$  \E \int_0^T \| |g^{n}-g|_{l_2}\|^2_{L_2} \, dt \to 0$. Let $u^n$ be the solution of the equation corresponding to the truncated data. Then by It\^o's formula one can easily check that
\begin{align*}
\E \int_0^T\|u^n_t-u_t\|^2_{H^{-1}}\, dt & \leq N \E  \|\xi -\xi^n\|_{H^{-1}}^2
\\
& + N \E \left( \int_0^T \|f-f^n\|^2_{H^{-1}}+ \sum_{k=1}^\infty\|g^k-g^{n,k}\|_{L_2}^2 \, dt \right)\to 0,
\end{align*}
as $n \to \infty$. Since for each $n\in \mathbb{N}$ we have 
$$
\E \|u^n \|^2 _{L_\infty(Q_T)} \ \leq N \E \left( 1+\|\xi \|_{L_\infty(Q)}^2 +|S_\mu(f,g)|^2 \right),
$$
the result follows by virtue of Lemma \ref{lem: lsc}. 
\end{proof}

\begin{theorem} \label{thm: smoothing PME}
Suppose that Assumptions \ref{as: coeff} and \ref{as: extra condition} are satisfied with $m>1$. Let {\color{black}$\mu \in \Gamma_d$} and let $u$ be the solution of \eqref{eq: main equation}. Then, for all $\rho \in (0,1)$ we have 
\begin{align}          
\E \|u \|^2 _{L_\infty((\rho, T) \times  Q)} \ &\leq \rho^{-\tilde{\theta}} N \E \left(\| \xi \|_{H^{-1}}^2 +{\color{black}|S_\mu(f,g)|^2} \right),
\end{align}
where $S_\mu(f,g)$ is as in Theorem \ref{thm: theorem boundedness},  $N$ is a constant depending only on  $C,  m, T, c,K,d, \mu$ and $|Q|$, and $\tilde{\theta}>0$  is a constant depending only on $ d, \mu$ and $m$.
\end{theorem}

\begin{proof}
\emph{Step 1:} In a first step we assume that $\xi \in  L_{m+1}(\Omega, L_{m+1}(Q))$ and that $|f_t(x)|$, $|g_t(x)|_{l_2}$ are bounded uniformly in $(\omega,t,x)$. 
Let $u^\varepsilon$ denote the unique solution of the problem \eqref{eq: single approximation}. By Lemma \ref{lem: regularity} $u^\varepsilon$ is a solution also in the sense of Definition \ref{def: definition L2}. Hence,  by Lemma \ref{thm: smoothing} we have 
\begin{equation}              \label{eq:K}
\E \|u^\varepsilon \|^2 _{L_\infty((\rho, T) \times  Q)} \ \leq \rho^{-\tilde{\theta}} N \E \left( 1+\| u^\varepsilon \|_{L_2(Q_T)}^2  +{\color{black}|\mathcal{S}_r(f,g)|^2} \right),
\end{equation}
with $N$ independent of $\varepsilon$. By It\^o's formula, Assumption \ref{as: coeff} \eqref{item: coercivity of phi} and Lemma \ref{lem: monotonicity} we have 
\begin{align}
\nonumber
&\|u^\varepsilon_t\|^2_{H^{-1}}+ \int_0^t \| u^\varepsilon\|^{m+1}_{L_{m+1}} \, ds 
\\
\label{eq:G}
&\leq N\left(1+ \|\xi\|^2_{H^{-1}}+\int_0^t\|u^\varepsilon\|^2_{H^{-1}}+\|f\|_{H^{-1}}^2+\sum_{k=1}^\infty\|g^k\|_{H^{-1}}^2 \, ds\right) +M_t,
\end{align}
for a local martingale, with $N$ independent of $\varepsilon$. After a localization argument and Gronwal's lemma one gets
$$
\E \int_0^T \| u^\varepsilon\|^{m+1}_{L_{m+1}} \, ds \leq 
N\E \left(1+ \|\xi\|^2_{H^{-1}}+\int_0^T (\|f\|_{H^{-1}}^2+\sum_{k=1}^\infty\|g^k\|_{H^{-1}}^2 ) \, ds\right).
$$
Plugging this in \eqref{eq:K} ($m+1>2$) gives the desired inequality.

\emph{Step 2:} For general $\xi, f$ and $g$, one can proceed as in the proof of Theorem \ref{thm: theorem boundedness}, this time choosing  $\xi^n \in L_{m+1}(\Omega; L_{m+1}(Q))$ such that $\lim_n\E\|\xi ^n- \xi\|_{H^{-1}}^2 =0$ and $\|\xi_n\|_{H^{-1}} \leq \|\xi\|_{H^{-1}}$ almost surely. This finishes the proof.
\end{proof}

\begin{remark} \label{rem}
As already seen, for any  $\xi \in L_2(\Omega; H^{-1})$, the corresponding solution $u$ of \eqref{eq: main equation} belongs to the space $ L_{m+1}(\Omega_T; L_{m+1}(Q))$. Consequently, there exists arbitrarily small $s>0$ such that 
$\E \|u_s \|_{L_{m+1}}^{m+1}< \infty$. By Remark \ref{rem:H10} the quantity $ \E \|\Phi(u^{\varepsilon}) \|_{L_2(s,T;H^1_0)}^2$ (where $u^{\varepsilon}$ is the solution of \eqref{eq: single approximation} starting at time $s$ from $u^{\varepsilon}_s=u_s$) can be controlled by $\E \|u_s\|^{m+1}_{L_{m+1}}$. Using this, one can use again the theory of monotone operators to show that the weak limit of $\Phi(u^\varepsilon)$ in $L_2(\Omega \times (s,T);H^1_0)$ coincides with $\Phi(u)$. In particular, the solution $u$ is strong on the time interval $(s,T)$, that  is, $\Phi(u_t)\in H^1_0(Q)$ for a.e. $(\omega,t) \in \Omega \times (s,T)$.  
\end{remark}

\appendix
\section{}

\begin{lemma} \label{lem:Appendix}
Let $Q \subset \bR^d$ be an open bounded set and let  $R  \in C^1 (\overline{Q} \times \bR)$ be such that there exist $N \in \bR$, $p \in [2, \infty)$ and $g \in L_p(Q)$  such that for all $(x,r) \in  \bar{Q} \times \bR$
\begin{equation}
 |R(x,r)|+| \nabla_x R(x,r)| \leq N+|g(x)||r|^{p-2} +N |r|^{p-1}.
 \end{equation}

Set 
$$
G(x,r):= \int_0^r R(x,s) \, ds,
$$
and let  $u \in H^1_0(Q)$ be such that 
\begin{equation}                      \label{eq:integrability-u}
\int_Q |u|^p \, dx + \int_Q | \nabla u |^2|u|^{p-2} \, dx < \infty.
\end{equation}
Then $G(\cdot, u) \in W^{1,1}_0(Q)$. 
\end{lemma}
\begin{proof}
Let us set 
\[
    R_n(x,r) := \left\{\begin{array}{lcr}
         R(x,r),  &\text{for}& \qquad |r| \leq n \\
         R (x,n),  &\text{for}& \qquad r > n \\
         R(x,-n),  &\text{for}& \qquad r < -n
        \end{array} \right.
  \]
  and 
  $$
  G_n(x,r) : = \int_0^r R_n(x,s) \, ds.
  $$
  It follows that $ \nabla_x G_n(x, r)$ and $\D_r  G_n(x, r)$ are continuous in $(x,r) \in \overline{Q} \times \bR$, and they satisfy with  some constant $N(n)$, for all $(x,r) \in \overline{Q} \times \bR$
  $$
  |\nabla_x G_n(x, r)| \leq N(n) |r|, \qquad |\D_r  G_n(x, r)| \leq N(n).
  $$ 
  Moreover, we have $G_n(x,0)=0$. 
  Hence, by approximating $u$ in $H^1_0(Q)$ with $u^m \in C^\infty_c(Q)$, one concludes easily that  $G_n(\cdot, u) \in W^{1,1}_0(Q)$. Notice that there exists a constant $N$, such that for all $n \in \mathbb{N}, \ (x,r) \in \overline{Q} \times \bR$ we have 
 \begin{enumerate}[(i)]
\item $|G_n(x,r)| \leq N(1+|g(x)|^p +|r|^{p})$,
\item $ | \nabla_x G_n(x,r)| \leq N(1+|g(x)|^p +|r|^{p}) $,
\item $|\D_rG_n(x,r)| \leq N(1+| g(x)||r|^{p-2}+|r|^{p-1})$.
\end{enumerate}
This implies by Young's inequality 
\begin{align*}
|G_n(x,u)| &\leq N(1+|g(x)|^p +|u|^{p}),
\\
|\nabla_x G_n(x,u)| + |\D_rG_n(x,u)| |\nabla u| &\leq N(1+| g(x)|^p+|u|^p+|u|^{p-2}|\nabla u|^2).
\end{align*} By Lebesgue's theorem on dominated convergence we have 
$G_n(\cdot, u)\to G(\cdot,u)$ and $ \nabla_x( G_n(\cdot, u)) \to \nabla_xG(\cdot, u)+ \D_rG(\cdot, u) \nabla_x u$
in $L_1(Q)$, and the claim follows since $ G_n(\cdot,u) \in W^{1,1}_0(Q)$ for all $n \in \mathbb{N}$.
\end{proof}

\section*{Acknowledgement}
B. Gess acknowledges financial support by the DFG through the CRC 1283 ``Taming uncertainty and profiting from randomness and low regularity in analysis, stochastics and their applications".

\end{document}